\documentclass[oneside,english]{amsart}
\usepackage[T1]{fontenc}
\usepackage[latin9]{inputenc}
\usepackage{geometry}
\usepackage{verbatim}
\usepackage{amstext}
\usepackage{amsthm}
\usepackage{amssymb}
 \usepackage{url}
 
\makeatletter
\numberwithin{equation}{section}
\numberwithin{figure}{section}
\theoremstyle{plain}
\newtheorem{thm}{\protect\theoremname}
  \theoremstyle{plain}
\newtheorem{lem}[thm]{\protect\lemmaname}
	\theoremstyle{plain}
\newtheorem{cor}[thm]{\protect\corollaryname}
  \theoremstyle{remark}
  \newtheorem{rem}[thm]{\protect\remarkname}
  \theoremstyle{definition}
  \newtheorem{example}[thm]{\protect\examplename}
\newtheorem{prop}[thm]{Proposition}

\makeatother

\usepackage{babel}
  \providecommand{\corollaryname}{Corollary}
  \providecommand{\examplename}{Example}
  \providecommand{\lemmaname}{Lemma}
  \providecommand{\remarkname}{Remark}
\providecommand{\theoremname}{Theorem}

\begin{document}

\title{Finite generating functions for the sum of digits sequence}

\author{C. Vignat and T. Wakhare}

\address{L.S.S. Supelec, Universite Paris Sud, Orsay, France and Department
of Mathematics, Tulane University, New Orleans, USA}

\address{University of Maryland, College Park, MD 20742, USA}

\begin{abstract}
We derive some new finite sums involving the sequence $s_{2}\left(n\right),$
the sum of digits of the expansion of $n$ in base $2.$ These functions
allow us to generalize some classical results obtained by Allouche,
Shallit and others.
\end{abstract}

\maketitle

\section{Introduction and main results}

The sum of digits sequence $s_{2}\left(n\right)$ counts the number
of ones in the binary expansion of $n,$ and starts, from $n=0,$
as follows:
\[
0,1,1,2,1,2,2,3,1,2,2,3,2,3,3,4\dots.
\]
This sequence has a strong redundancy expressed by the identities
\[
\begin{cases}
s_{2}\left(2n\right) & =s_{2}\left(n\right),\\
s_{2}\left(2n+1\right) & =s_{2}\left(n\right)+1,
\end{cases}
\]
which suggests that series of the form 
\[
\sum_{n}s_{2}\left(n\right)f\left(n\right)\thinspace\thinspace\text{or}\thinspace\thinspace\sum_{n}\left(-1\right)^{s_{2}\left(n\right)}f\left(n\right)
\]
can be computed explicitly for some simple sequences 
$\left\{ f\left(n\right)\right\} .$
Many results have been obtained in this direction, such as the beautiful
\begin{equation}
\sum_{n=1}^{+\infty}s_{2}\left(n\right)\left(\frac{1}{n^{p}}-\frac{1}{\left(n+1\right)^{p}}\right)=\frac{1-2^{1-p}}{1-2^{-p}}\zeta\left(p\right)=\frac{1}{1-2^{-p}}\eta\left(p\right)\label{eq:zeta-2}
\end{equation}
for any complex number $p$ with $\Re\left(p\right)>0,$ where $\zeta\left(p\right)$
is Riemann's zeta function $\zeta\left(p\right)=\sum_{n\ge1}\frac{1}{n^{p}}$
and $\eta\left(p\right)$ is Dirichlet's eta function $\eta\left(p\right)=\sum_{n\ge1}\frac{\left(-1\right)^{n-1}}{n^{p}}$. This identity was proved by Allouche and Shallit \cite{Allouche Shallit-1}.
Another classic result is
\begin{equation}
\label{log2}
\sum_{n=1}^{+\infty}\frac{s_{2}\left(n\right)}{n\left(n+1\right)}=2\log2,
\end{equation}
a proof of which was one of the problems at the 1981 Putnam competition \cite{Putnam}.  
This result was generalized by Allouche and Shallit \cite{Allouche Shallit-1} to an arbitrary base $b$ as follows:
\begin{equation}
\label{logb}
\sum_{n=1}^{+\infty}\frac{s_{b}\left(n\right)}{n\left(n+1\right)}=\frac{b}{b-1}\log b,
\end{equation}
where $s_{b}\left(n\right)$ is the sum of all digits in the expansion
of $n$ in base $b.$

The sum of digits sequence is intimately connected to computer science and various aspects of discrete mathematics. The sequence $\left\{ (-1)^{s_2(n)} \right\}$ is the prototype for an \textit{automatic sequence.}
%, which is a finite automaton accepting the base-$b$ digits of a number $n$ as input. 
These sequences, fundamental to the study of `combinatorics on words', have been extensively explored by Allouche and Shallit  \cite{Allouche Shallit book}. Any method developed to address finite sums involving $s_2(n)$ could potentially transfer to other automatic sequences. %Add some info on the prouhet-terry-morse problem as justification?

Moreover, the Prouhet-Thue-Morse sequence $\left\{t_n\right\}$ defined by $t_n = s_2\left(n\right) \mod 2$ appears to be ubiquitous, as underlined in the famous article \cite{Allouche ubiquitous} by Allouche and Shallit. It is regular yet not ultimately periodic, which means that it can serve as a constructive element of many proofs.

Another motivation to study the sequence $s_2\left(n\right)$ is its link with the valuation sequence $\nu_2\left(n\right),$ the valuation of $n$ being defined as the exponent of the highest power of $2$ that divides $n$; this function is important in number theory. Identity \eqref{eq:2-adic} in Thm \ref{thm:2-adic} shows that some results about the sequence $s_2\left(n\right)$ can potentially transfer to the valuation sequence $\nu_2\left(n\right).$

Interestingly, the recursive properties of $s_b(n)$ also have an unexpected connection to $L$-functions from analytic number theory. An $L$-function is defined as $\sum_{n=1}^\infty \frac{\chi(n)}{n^s}$, where $\chi$ is a Dirichlet character which is periodic with period $k$. It satisfies recurrences of the type $\chi(nk+r)=\chi(r)$ , which resemble the recurrence $s_b(nb+r)=s_b(n)+r$. Therefore, any methods developed to address finite sums involving $s_b(n)$ could potentially lead to the discovery of new finite sums involving Dirichlet characters or finite $L$-functions; both applications are of fundamental importance to analytic number theorists.

Finally, identity \eqref{eq:zeta-2} expresses the remarkable result that a variation of the Dirichlet series related to the sequence $s_2\left(n\right)$ is proportional to the Riemann zeta function: given such a remarkable result, one may wonder if its elegance extends to a wider class of functions such as the Hurwitz zeta function - this is indeed the case as shown in Thm \ref{thm:Thm3}. Another relevant question is to wonder if the truncation of this series to a finite sum has an attractive closed form; this is positively answered in Thm \ref{thm:Thm1-1}. 

As a consequence, in this article, we are interested in finite sums that involve the sequence
$s_{b}\left(n\right),$ of the form
$$\sum_{n=0}^{b^p-1} s_{b}\left(n\right)f\left(n\right),$$
the idea being that summing over a fundamental
interval of the form $\left[0,b^{p}-1\right]$ keeps some symmetry
that allows us to find explicit values for these series for a variety
of sequences $\left\{ f\left(n\right)\right\} .$ 
In particular, when expressed in base $b$, 
the interval $\left[0,b^{p}-1\right]$ runs over all possible words of the form $\{w_1w_2\ldots w_p:w_i \in \{0,\ldots,b-1\}\}$.
%For the remainder of this article, we will address sums of the form $$\sum_{n=0}^{b^p-1} s_{b}\left(n\right)f\left(n\right).$$

The second main idea of this paper is to replace sequences $f\left(n\right)$ by functions $f\left(n+z\right)$ of an additional real variable $z$, noticing that many tools that allow to compute these sums - such as telescoping - still apply; the extra variable $z$ allows us (see for example the proof of Thm \ref{thm:Thm1-1}) to extend any identity based on a function $f$ to a family of identities based on functions related to $f$ by integral transforms. For example, the generalization of the series in \eqref{eq:zeta-2} to a function of a variable $x$ allows us to show in Thm \ref{thm:A-Lambert-series} that the obtained function can be related by a Laplace transform to a Lambert series type expression of the generating function of the sequence $s_b\left(n\right)$.

Our first result is the finite sum from Thm \ref{thm:Thm1-1}: for
all $\alpha>0,\,\,\alpha \ne 1$ and $z\ge0,$
%\begin{align}
%\sum_{n=1}^{2^{p}-1}s_{2}\left(n\right)\left(\frac{1}{\left(z+n\right)^{\alpha}}-\frac{1}{\left(z+n+1\right)^{\alpha}}\right) & =\sum_{l=0}^{p-1}\frac{1}{\left(2^{l+1}\right)^{\alpha}}\left[\zeta\left(\alpha,\frac{1}{2}+\frac{z}{2^{l+1}}\right)-\zeta\left(\alpha,1+\frac{z}{2^{l+1}}\right)\right.\label{eq:zeta-1}\\
% & \left.-\zeta\left(\alpha,\frac{1}{2}+\frac{z+2^{p}}{2^{l+1}}\right)+\zeta\left(\alpha,1+\frac{z+2^{p}}{2^{l+1}}\right)\right];\nonumber 
%\end{align}
\begin{align*}
\sum_{n=1}^{b^{p}-1}s_{b}\left(n\right)\left(\frac{1}{\left(z+n\right)^{\alpha}}-\frac{1}{\left(z+n+1\right)^{\alpha}}\right)
=&\sum_{l=0}^{p-1}\frac{1}{b^{\alpha l}}\left[\zeta\left(\alpha,1+\frac{z}{b^{l}}\right)-\zeta\left(\alpha,1+\frac{z+b^{p}}{b^{l}}\right)\right]\\
-&\sum_{l=1}^{p}\frac{b}{b^{\alpha l}}\left[\zeta\left(\alpha,1+\frac{z}{b^{l}}\right)-\zeta\left(\alpha,1+\frac{z+b^{p}}{b^{l}}\right)\right]. 
\end{align*}
It is a common generalization of both (\ref{eq:zeta-2}) and (\ref{logb}) which also depends on the extra variable $z$ and has a finite upper summation limit. The limit case $p\to\infty$
at $z=0$ is shown to recover (\ref{eq:zeta-2}), while the special case $\alpha=1$ is given by identity \eqref{eq:J2p-1 2}.
%Functionalizing (\ref{eq:zeta-2}) by adding the parameter $z$ allows us to apply nontrivial transformations such as the Laplace transform to our formulae, leading to even more finite sums.
We also prove, at the end of Section 6, the following explicit expression for a finite zeta function: for $\alpha > 2$ and $z\ge 0,$

\begin{align}
\sum_{n=1}^{b^{p}-1}\frac{s_{b}\left(n\right)}{\left(n+z\right)^{\alpha}} & =\sum_{l=0}^{p-1}\left[\zeta_{2}\left(\alpha,z+b^{l},\left(1,b^{l}\right)\right)-\zeta_{2}\left(\alpha,z+b^{l}+b^{p},\left(1,b^{l}\right)\right)\right]\\
 & -b\sum_{l=1}^{p}\left[\zeta_{2}\left(\alpha,z+b^{l},\left(1,b^{l}\right)\right)-\zeta_{2}\left(\alpha,z+b^{l}+b^{p},\left(1,b^{l}\right)\right)\right]\nonumber
\end{align}

and its asymptotic version

\begin{align}
\sum_{n=1}^{+\infty}\frac{s_{b}\left(n\right)}{\left(n+z\right)^{\alpha}} 
 & =-z\zeta\left(\alpha,z+1\right)+\zeta\left(\alpha-1,z+1\right)+\left(1-b\right)\sum_{l=1}^{+\infty}\zeta_{2}\left(\alpha,z+b^{l},\left(1,b^{l}\right)\right),
\end{align}
where $\zeta_{2}\left(\alpha,z,\left(a,b\right)\right)$ is the Barnes zeta function $
\zeta_{2}\left(\alpha,z,\left(a,b\right)\right)=\sum_{m_{1},m_{2} \ge 0}\left(z+a_{1}m_{1}+a_{2}m_{2}\right)^{-\alpha}.
$ We also deduce an expression for the special case $\alpha =2$, given by Cor \ref{cor30}.

Our second main result is the general formula
\begin{equation}
\sum_{n=0}^{2^{N+1}-1}\left(-1\right)^{s_{2}\left(n\right)}f\left(x+n\right)=\left(-1\right)^{N+1}2^{\frac{N\left(N+1\right)}{2}}\thinspace\sum_{k=0}^{2^{N+1}-N-2}p_{k}^{\left(N\right)}\Delta^{N+1}f\left(x+k\right),\label{eq:Sn Thm 2-1}
\end{equation}
where $\Delta$ is the forward finite difference operator, and the
positive weights $p_{k}^{\left(N\right)}$ are characterized in Thm
\ref{thm:Thm1}.

Our final result is the following formal infinite sum for any sequence
$\left\{ g\left(n\right)\right\}$, as given in Thm \ref{thm:Sum base 2}:
\[
\sum_{n\ge1}s_{2}\left(n\right)g\left(n\right)=\sum_{n,k\ge0}\sum_{l=0}^{2^{k}-1}g\left(2^{k+1}n+2^{k}+l\right).
\]
Setting $g(n)=0$ for $n\geq N_0$ recovers a finite sum for $s_2(n)$. Its extension to an arbitrary base $b$ is given in Thm \ref{thm:Sum base b}
and its explicit finite counterpart is
\[
\sum_{n=1}^{2^{p}-1}s_{2}\left(n\right)g\left(n\right)=\sum_{n=1}^{2^{p-1}-1}\sum_{k=0}^{p-l_{2n+1}}\sum_{l=0}^{2^{k}-1}g\left(n2^{k+1}+2^{k}+l\right),
\]
as described in Thm \ref{thm:Finite sum}, where $l_n$ is the number of digits in the binary representation of $n$.

\section{About the series $\sum_{n=1}^{N}\frac{s_{b}\left(n\right)}{\left(x+n\right)\left(x+n+1\right)}$}

Let us write
\[
J_{N}^{\left(b\right)}\left(x\right):=\sum_{n=1}^{N}\frac{s_{b}\left(n\right)}{\left(x+n\right)\left(x+n+1\right)}.
\]
We will also use the auxiliary function
\[
\gamma_{N}\left(x\right):=\sum_{n=0}^{N}\frac{1}{\left(x+2n+1\right)\left(x+2n+2\right)}.
\]
A simple computation shows that it can be expressed as the difference

\begin{eqnarray*}
\gamma_{N}\left(x\right) & = & \beta\left(x+1\right)-\beta\left(x+2N+3\right),
\end{eqnarray*}
where
\[
\beta\left(x\right):=\frac{1}{2}\left[\psi\left(\frac{x+1}{2}\right)-\psi\left(\frac{x}{2}\right)\right]
\]
is Stirling's beta function and $\psi\left(z\right)$ is the digamma function. We also denote the Hurwitz zeta function by
\[
\zeta\left(\alpha,z\right)=\sum_{n\ge0}\frac{1}{\left(z+n\right)^{\alpha}}.
\]
%and its truncated version by
%\[
%\tilde{\zeta}\left(\alpha,z\right)=\sum_{n\ge1}\frac{1}{\left(z+n\right)^{\alpha}}=\zeta\left(\alpha,z\right)-\frac{1}{z^{\alpha}}.
%\]
Our first result is as follows:
\begin{thm}
\label{thm:Thm1-1}For all $\alpha>0,\,\,\alpha \ne 1$ and $z\ge0,$ we have
\begin{align*}
\sum_{n=1}^{b^{p}-1}s_{b}\left(n\right)\left(\frac{1}{\left(z+n\right)^{\alpha}}-\frac{1}{\left(z+n+1\right)^{\alpha}}\right)
=&\sum_{l=0}^{p-1}\frac{1}{b^{\alpha l}}\left[\zeta\left(\alpha,1+\frac{z}{b^{l}}\right)-\zeta\left(\alpha,1+\frac{z+b^{p}}{b^{l}}\right)\right]\\
-&\sum_{l=1}^{p}\frac{b}{b^{\alpha l}}\left[\zeta\left(\alpha,1+\frac{z}{b^{l}}\right)-\zeta\left(\alpha,1+\frac{z+b^{p}}{b^{l}}\right)\right]. 
%=&\sum_{l=0}^{p-1}\frac{1}{b^{\alpha l}}\left[\tilde{\zeta}\left(\alpha,\frac{z}{b^{l}}\right)-\tilde{\zeta}\left(\alpha,\frac{z+b^{p}}{b^{l}}\right)\right]\\
%-&\sum_{l=1}^{p}\frac{b}{b^{\alpha l}}\left[\zeta\left(\alpha,1+\frac{z}{b^{l}}\right)-\zeta\left(\alpha,1+\frac{z+b^p}{b^{l}}\right)\right].
%=\sum_{l=0}^{p-1}\frac{1}{\left(b^{l+1}\right)^{\alpha}}\sum_{k=1}^{b-1}k\left[\zeta\left(\alpha,\frac{k}{b}+\frac{z}{b^{l+1}}\right)-\zeta\left(\alpha,\frac{k}{b}+\frac{z}{b^{l+1}}+b^{p-l-1}\right)\right.\\
% & -\left.\zeta\left(\alpha,\frac{k+1}{b}+\frac{z}{b^{l+1}}\right)+\zeta\left(\alpha,\frac{k+1}{b}+\frac{z}{b^{l+1}}+b^{p-l-1}\right)\right].
\end{align*}
The case $\alpha = 1$ is given, for all $z\ge0$, by
\begin{align*}
\sum_{n=1}^{b^{p}-1}\frac{s_{b}\left(n\right)}{\left(z+n\right)\left(z+n+1\right)} & =
\sum_{l=0}^{p-1}\frac{1}{b^{l}}\left[\psi\left(1+\frac{z+b^{p}}{b^{l}}\right)-\psi\left(1+\frac{z}{b^{l}}\right)\right]
-\sum_{l=1}^{p}\frac{b}{b^{l}}\left[\psi\left(1+\frac{z+b^{p}}{b^{l}}\right)-\psi\left(1+\frac{z}{b^{l}}\right)\right].
%\sum_{n=1}^{b^{p}-1}\frac{s_{b}\left(n\right)}{\left(z+n\right)\left(z+n+1\right)} & =\sum_{l=0}^{p-1}\frac{1}{b^{l+1}}\sum_{k=1}^{b-1}k\left[\psi\left(\frac{k}{b}+\frac{z}{b^{l+1}}+b^{p-l-1}\right)-\psi\left(\frac{k}{b}+\frac{z}{b^{l+1}}\right)\right.\\
% & -\left.\psi\left(\frac{k+1}{b}+\frac{z}{b^{l+1}}+b^{p-l-1}\right)+\psi\left(\frac{k+1}{b}+\frac{z}{b^{l+1}}\right)\right].
\end{align*}
\end{thm}
The binary case $b=2$ has the following simple form:
\begin{cor}
For all $\alpha \ne 1$ and $z\ge0$,
\begin{align}
\sum_{n=1}^{2^{p}-1}s_{2}\left(n\right)\left(\frac{1}{\left(z+n\right)^{\alpha}}-\frac{1}{\left(z+n+1\right)^{\alpha}}\right) & =\sum_{l=0}^{p-1}\frac{1}{\left(2^{l+1}\right)^{\alpha}}\left[\zeta\left(\alpha,\frac{1}{2}+\frac{z}{2^{l+1}}\right)-\zeta\left(\alpha,1+\frac{z}{2^{l+1}}\right)\right.\label{eq:zeta}\\
 & \left.-\zeta\left(\alpha,\frac{1}{2}+\frac{z+2^{p}}{2^{l+1}}\right)+\zeta\left(\alpha,1+\frac{z+2^{p}}{2^{l+1}}\right)\right].\nonumber 
\end{align}
The special case $\alpha = 1$ is given, for all $z\ge0$, by
\begin{align}
\sum_{n=1}^{2^{p}-1}\frac{s_{2}\left(n\right)}{\left(z+n\right)\left(z+n+1\right)} & =\sum_{l=0}^{p-1}\frac{1}{2^{l+1}}\left[\psi\left(\frac{z}{2^{l+1}}+1\right)-\psi\left(\frac{z}{2^{l+1}}+\frac{1}{2}\right)\right.\label{eq:J2p-1 2}\\
 & \left.-\psi\left(\frac{z+2^p}{2^{l+1}}+1\right)+\psi\left(\frac{z+2^p}{2^{l+1}}+\frac{1}{2}\right)\right].\nonumber
\end{align}
\end{cor}

\begin{rem}
The right-hand side of \eqref{eq:zeta} can also be expressed as the double sum 
\[
2^{\alpha}\sum_{l=0}^{p-1}\sum_{n\ge1}\left(\frac{\left(-1\right)^{n}}{\left(z+n \cdot 2^{l+1}\right)^{\alpha}}-\frac{\left(-1\right)^{n}}{\left(z+2^{p}+n \cdot 2^{l+1}\right)^{\alpha}}\right).
\]
\end{rem}

\begin{proof}
As for most other results in this section, the main part of the proof of Thm \ref{thm:Thm1-1} is provided only in the case $b=2$ for two reasons: first, its extension to an arbitrary base $b$ is elementary and does not require any extra tools; second, the binary case is easier to write and reveals the mechanism of the proof more directly. 
Let us thus assume $b=2.$ We first need the following recurrence for the function $J_{N}^{\left(2\right)}\left(x\right)$, which we write as $J_{N}\left(x\right)$ for simplicity.
\begin{lem}
A recurrence for $J_{N}\left(x\right)$ is 
\begin{equation}
J_{2N+1}\left(x\right)=\frac{1}{2}J_{N}\left(\frac{x}{2}\right)+\gamma_{N}\left(x\right)\label{eq:recurrence J2n-1}.
\end{equation}
As a consequence, for all integers $N$ and $p,$ 
\begin{equation}
J_{N \cdot 2^{p}-1}\left(x\right)=\frac{1}{2^{p}}J_{N-1}\left(\frac{x}{2^{p}}\right)+\sum_{l=0}^{p-1}\frac{1}{2^{l}}\gamma_{N \cdot 2^{p-1-l}-1}\left(\frac{x}{2^{l}}\right).\label{eq:JN2p-1}
\end{equation}
The case $N=1$ is
\begin{align}
J_{2^{p}-1}\left(x\right) & =\sum_{l=0}^{p-1}\frac{1}{2^{l}}\left[\beta\left(\frac{x}{2^{l}}+1\right)-\beta\left(\frac{x}{2^{l}}+2^{p-l}+1\right)\right]\label{eq:J2p-1}
\end{align}
or equivalently \eqref{eq:J2p-1 2}.
%\begin{align}
%\sum_{n=1}^{2^{p}-1}\frac{s_{2}\left(n\right)}{\left(x+n\right)\left(x+n+1\right)} & =\sum_{l=0}^{p-1}\frac{1}{2^{l+1}}\left[\psi\left(\frac{x}{2^{l+1}}+\frac{1}{2}\right)-\psi\left(\frac{x}{2^{l+1}}+1\right)\right.\label{eq:J2p-1 2}\\
% & \left.-\psi\left(\frac{x}{2^{l+1}}+2^{p-l-1}+\frac{1}{2}\right)+\psi\left(\frac{x}{2^{l+1}}+2^{p-l-1}+1\right)\right]\nonumber 
%\end{align}
\end{lem}

\begin{proof}
Using the recurrences
\[
s_{2}\left(2n\right)=s_{2}\left(n\right),\thinspace\thinspace s_{2}\left(2n+1\right)=s_{2}\left(n\right)+1,
\]
we deduce
\begin{eqnarray*}
J_{2N-1}\left(x\right) & = & \sum_{n=1}^{N-1}s_{2}\left(n\right)\left[\frac{1}{\left(x+2n\right)\left(x+2n+1\right)}+\frac{1}{\left(x+2n+1\right)\left(x+2n+2\right)}\right]+\gamma_{N-1}\left(x\right)\\
 & = & \sum_{n=1}^{N-1}\frac{2s_{2}\left(n\right)}{\left(x+2n\right)\left(x+2n+2\right)}+\gamma_{N-1}\left(x\right)\\
 & = & \frac{1}{2}\sum_{n=1}^{N-1}\frac{s_{2}\left(n\right)}{\left(\frac{x}{2}+n\right)\left(\frac{x}{2}+n+1\right)}+\gamma_{N-1}\left(x\right)=\frac{1}{2}J_{N-1}\left(\frac{x}{2}\right)+\gamma_{N-1}\left(x\right).
\end{eqnarray*}
Iterating this identity gives
\[
J_{N \cdot 2^{p}-1}\left(x\right)=\frac{1}{2^{p}}J_{N-1}\left(\frac{x}{2^{p}}\right)+\sum_{l=0}^{p-1}\frac{1}{2^{l}}\gamma_{N \cdot 2^{p-1-l}-1}\left(\frac{x}{2^{l}}\right).
\]
Choosing $N=1$ and noticing that $J_{0}\left(x\right)=0$ gives
(\ref{eq:J2p-1}).
\end{proof}
The proof of Thm \ref{thm:Thm1-1} starts with the Mellin transform \cite[Entry 2.2.4.25]{Prudnikov-1}
\[
\int_{0}^{+\infty}\frac{x^{\alpha-1}}{\left(x+z\right)\left(x+z+1\right)}dx=\frac{\pi}{\sin\pi\alpha}\left(z^{\alpha-1}-\left(z+1\right)^{\alpha-1}\right),
\]
with $0<\alpha<1$ and $z\ge0$; replacing $z$ by $z+n$ produces
\[
\int_{0}^{+\infty}\frac{x^{\alpha-1}}{\left(x+z+n\right)\left(x+z+n+1\right)}dx=\frac{\pi}{\sin\pi\alpha}\left(\left(z+n\right)^{\alpha-1}-\left(z+n+1\right)^{\alpha-1}\right).
\]
Replacing $x$ by $x+z$ in the left-hand side of (\ref{eq:J2p-1 2}),
multiplying by $x^{\alpha-1}$ and integrating yields
\[
\frac{\pi}{\sin\pi\alpha}\sum_{n=1}^{2^{p}-1}s_{2}\left(n\right)\left(\frac{1}{\left(z+n\right)^{1-\alpha}}-\frac{1}{\left(z+n+1\right)^{1-\alpha}}\right).
\]
Next, using \cite[Entry 2.3.1.5]{Prudnikov},
\[
\int_{0}^{+\infty}x^{\alpha-1}\left[\psi\left(x+a\right)-\psi\left(x+b\right)\right]dx=\frac{\pi}{\sin\pi\alpha}\left[\zeta\left(1-\alpha,b\right)-\zeta\left(1-\alpha,a\right)\right],
\]
for $a,b>0$ and $0<\alpha<1$, allows us to deduce
\[
\int_{0}^{+\infty}x^{\alpha-1}\left[\psi\left(\frac{x}{2^{l+1}}+a\right)-\psi\left(\frac{x}{2^{l+1}}+b\right)\right]dx=\left(2^{l+1}\right)^{\alpha}\frac{\pi}{\sin\pi\alpha}\left[\zeta\left(1-\alpha,b\right)-\zeta\left(1-\alpha,a\right)\right].
\]
Replacing $x$ by $x+z$ in the right-hand side and integrating after
multiplication by $x^{\alpha-1}$ produces
\begin{align*}
\sum_{l=0}^{p-1}\frac{1}{2^{l+1}}\int_{0}^{+\infty}x^{\alpha-1}\left[\psi\left(\frac{x+z}{2^{l+1}}+1\right)-\psi\left(\frac{x+z}{2^{l+1}}+\frac{1}{2}\right)-\psi\left(\frac{x+z}{2^{l+1}}+2^{p-l-1}+1\right)+\psi\left(\frac{x+z}{2^{l+1}}+2^{p-l-1}+\frac{1}{2}\right)\right]dx\\
=\frac{\pi}{\sin\pi\alpha}\sum_{l=0}^{p-1}\frac{1}{\left(2^{l+1}\right)^{1-\alpha}}\left[\zeta\left(1-\alpha,\frac{1}{2}+\frac{z}{2^{l+1}}\right)-\zeta\left(1-\alpha,1+\frac{z}{2^{l+1}}\right)\right.\\
\left.-\zeta\left(1-\alpha,2^{p-l-1}+\frac{1}{2}+\frac{z}{2^{l+1}}\right)+\zeta\left(1-\alpha,2^{p-l-1}+1+\frac{z}{2^{l+1}}\right)\right]
\end{align*}
and we deduce, for all $z\ge0$ and for $0<\alpha<1,$
\begin{align*}
\sum_{n=1}^{2^{p}-1}s_{2}\left(n\right)\left(\frac{1}{\left(z+n\right)^{1-\alpha}}-\frac{1}{\left(z+n+1\right)^{1-\alpha}}\right) & =\sum_{l=0}^{p-1}\frac{1}{\left(2^{l+1}\right)^{1-\alpha}}\left[\zeta\left(1-\alpha,\frac{1}{2}+\frac{z}{2^{l+1}}\right)-\zeta\left(1-\alpha,1+\frac{z}{2^{l+1}}\right)\right.\\
 & \left.-\zeta\left(1-\alpha,\frac{1}{2}+\frac{z+2^{p}}{2^{l+1}}\right)+\zeta\left(1-\alpha,1+\frac{z+2^{p}}{2^{l+1}}\right)\right].
\end{align*}
Replacing $\alpha$ by $1-\alpha$ produces identity \eqref{eq:zeta} for any $0<\alpha<1.$ 

Taking the derivative with respect to the variable $z$ of the left-hand
side of (\ref{eq:zeta}) replaces $\alpha$ by $\alpha+1$ with a
prefactor $\left(-\alpha\right).$ The representation of the Hurwitz
zeta function \cite[25.11.8]{NIST}
\[
\zeta\left(\alpha,\frac{w+1}{2}\right)-\zeta\left(\alpha,\frac{w}{2}+1\right)=2^{\alpha}\sum_{n\ge1}\frac{\left(-1\right)^{n}}{\left(w+n\right)^{\alpha}}
\]
shows that computing the derivative with respect to $z$ of the right-hand
side of (\ref{eq:zeta}) produces exactly the same effect. Hence $\alpha\in\left(0,1\right)$
can be replaced by $\alpha+1$ so that (\ref{eq:zeta}) holds in fact
for all non-integer $\alpha>0.$ By continuity, it holds for all $\alpha>0$ such that $\alpha \ne 1.$

In the case of an arbitrary base $b,$ the previous result can easily be shown to extend to the following identity:
\begin{align*}
\sum_{n=1}^{b^{p}-1}s_{b}\left(n\right)\left(\frac{1}{\left(z+n\right)^{\alpha}}-\frac{1}{\left(z+n+1\right)^{\alpha}}\right) & =\sum_{l=0}^{p-1}\frac{1}{\left(b^{l+1}\right)^{\alpha}}\sum_{k=1}^{b-1}k\left[\zeta\left(\alpha,\frac{k}{b}+\frac{z}{b^{l+1}}\right)-\zeta\left(\alpha,\frac{k}{b}+\frac{z}{b^{l+1}}+b^{p-l-1}\right)\right.\\
 & -\left.\zeta\left(\alpha,\frac{k+1}{b}+\frac{z}{b^{l+1}}\right)+\zeta\left(\alpha,\frac{k+1}{b}+\frac{z}{b^{l+1}}+b^{p-l-1}\right)\right].
\end{align*}
Then denoting
\[
\alpha_{k}=\zeta\left(\alpha,\frac{k}{b}+\frac{z}{b^{l+1}}\right)-\zeta\left(\alpha,\frac{k}{b}+\frac{z}{b^{l+1}}+b^{p-l-1}\right),
\]
the inner sum in Thm \ref{thm:Thm1-1} reads
\[
\sum_{k=1}^{b-1}k\left[\alpha_{k}-\alpha_{k+1}\right] =\sum_{k=1}^{b}\alpha_{k}-b\alpha_{b}.
\]
Moreover,
\begin{align*}
\sum_{k=1}^{b}\alpha_{k} & =\sum_{k=1}^{b}\zeta\left(\alpha,\frac{k}{b}+\frac{z}{b^{l+1}}\right)-\zeta\left(\alpha,\frac{k}{b}+\frac{z}{b^{l+1}}+b^{p-l-1}\right)\\
 & =\sum_{k=1}^{b}\zeta\left(\alpha,\frac{k+z'}{b}\right)-\zeta\left(\alpha,\frac{k+z''}{b}\right)
\end{align*}
with
\[
z'=\frac{z}{b^{l}}\thinspace\thinspace \text{and} \thinspace\thinspace z''=\frac{z}{b^{l}}+b^{p-l}.
\]
This expression can be simplified using the multiplication formula for the Hurwitz zeta function \cite[25.11.15]{NIST}
\begin{equation}
b^{\alpha}\zeta\left(\alpha,z\right)-\left(\frac{b}{z}\right)^{\alpha}=\sum_{k=1}^{b}\zeta\left(\alpha,\frac{z+k}{b}\right).\label{eq:multiplication}
\end{equation}
By applying this multiplication formula, we obtain
\begin{align*}
\sum_{k=1}^{b}\alpha_{k} & =\left(b^{\alpha}\zeta\left(\alpha,z'\right)-\left(\frac{b}{z'}\right)^{\alpha}\right)-\left(b^{\alpha}\zeta\left(\alpha,z''\right)-\left(\frac{b}{z''}\right)^{\alpha}\right).
\end{align*}
Noticing that
\[
\sum_{n\ge1}\frac{1}{\left(z+n\right)^{\alpha}}=\zeta\left(\alpha,z\right)-\frac{1}{z^{\alpha}}=\zeta\left(\alpha,z+1\right),
\]
this sum reads equivalently
\[
\sum_{k=1}^{b}\alpha_{k}=b^{\alpha}\left[\zeta\left(\alpha,1+\frac{z}{b^{l}}\right)-\zeta\left(\alpha,1+\frac{z+b^{p}}{b^{l}}\right)\right].
\]
We deduce
\begin{align*}
\sum_{k=1}^{b-1}k\left[\alpha_{k}-\alpha_{k+1}\right] & =b^{\alpha}\left[\zeta\left(\alpha,1+\frac{z}{b^{l}}\right)-\zeta\left(\alpha,1+\frac{z+b^{p}}{b^{l}}\right)\right]\\
 & -b\left[\zeta\left(\alpha,1+\frac{z}{b^{l+1}}\right)-\zeta\left(\alpha,1+\frac{z}{b^{l+1}}+b^{p-l-1}\right)\right].
\end{align*}
Finally, the right-hand side sum in Thm \ref{thm:Thm1-1} is
\begin{align*}
\sum_{l=0}^{p-1}\frac{1}{b^{\alpha l}}\left[\zeta\left(\alpha,1+\frac{z}{b^{l}}\right)-\zeta\left(\alpha,1+\frac{z+b^{p}}{b^{l}}\right)\right]\\
-\sum_{l=0}^{p-1}\frac{b}{b^{\left(l+1\right)\alpha}}\left[\zeta\left(\alpha,1+\frac{z}{b^{l+1}}\right)-\zeta\left(\alpha,1+\frac{z}{b^{l+1}}+b^{p-l-1}\right)\right]
\end{align*}
and can be simplified to the final result
%begin{align}
\[
\sum_{l=0}^{p-1}\frac{1}{b^{\alpha l}}\left[\zeta\left(\alpha,1+\frac{z}{b^{l}}\right)-\zeta\left(\alpha,1+\frac{z+b^{p}}{b^{l}}\right)\right]\\
-\sum_{l=1}^{p}\frac{b}{b^{l\alpha}}\left[\zeta\left(\alpha,1+\frac{z}{b^{l}}\right)-\zeta\left(\alpha,1+\frac{z+b^{p}}{b^{l}}\right)\right].
\]
%\\
%=\left(1-b\right)\sum_{l=1}^{p-1}\frac{1}{b^{\alpha l}}\left[\zeta\left(\alpha,1+\frac{z}{b^{l}}\right)-\zeta\left(\alpha,1+\frac{z+b^{p}}{b^{l}}\right)\right]\\
%+\left[\zeta\left(\alpha,1+z\right)-\zeta\left(\alpha,1+z+b^{p}\right)\right]\\
%-\frac{b}{b^{p\alpha}}\left[\zeta\left(\alpha,1+\frac{z}{b^{p}}\right)-\zeta\left(\alpha,2+\frac{z}{b^{p}}\right)\right]
%\end{align*}
The case $b=2$ is equivalent to, but more involved than that in Corollary 2.

In the case $\alpha=1,$ we use the multiplication formula for the
digamma function \cite[5.15.7]{NIST}
\[
\psi\left(bz\right)=\frac{1}{b}\sum_{k=0}^{b-1}\psi\left(z+\frac{k}{b}\right) + \log b.
\]
Similarly to the above computation, we obtain
\[
\sum_{k=1}^{b-1}k\left[\alpha_{k}-\alpha_{k+1}\right]=\sum_{k=1}^{b}\alpha_{k}-b\alpha_{b}
\]
with
\[
\alpha_{k}=\psi\left(\frac{k}{b}+\frac{z+b^{p}}{b^{l+1}}\right)-\psi\left(\frac{k}{b}+\frac{z}{b^{l+1}}\right).
\]
Therefore
\[
\sum_{k=1}^{b}\alpha_{k}=\sum_{k=1}^{b}\psi\left(\frac{k}{b}+z'\right)-\psi\left(\frac{k}{b}+z''\right),
\]
with
\[
z'=\frac{z+b^{p}}{b^{l+1}},\thinspace\thinspace z''=\frac{z}{b^{l+1}}.
\]
Since
\begin{align*}
\sum_{k=1}^{b}\psi\left(z+\frac{k}{b}\right) & =\sum_{k=0}^{b-1}\psi\left(z+\frac{k}{b}\right)+\psi\left(z+1\right)-\psi\left(z\right)\\
 & =\sum_{k=0}^{b-1}\psi\left(z+\frac{k}{b}\right)+\frac{1}{z},
\end{align*}
we deduce
\begin{align*}
\sum_{k=1}^{b}\alpha_{k} & =\frac{1}{z'}-\frac{1}{z''}+b\psi\left(bz'\right)-b\psi\left(bz''\right)\\
 & =\frac{b^{l+1}}{z+b^{p}}-\frac{b^{l+1}}{z}+b\psi\left(\frac{z}{b^{l}}+b^{p-l}\right)-b\psi\left(\frac{z}{b^{l}}\right)\\
 & =b\left[\frac{b^{l}}{z+b^{p}}+\psi\left(\frac{z}{b^{l}}+b^{p-l}\right)-\frac{b^{l}}{z}-b\psi\left(\frac{z}{b^{l}}\right)\right]\\
 & =b\left[\psi\left(1+\frac{z+b^{p}}{b^{l}}\right)-\psi\left(1+\frac{z}{b^{l}}\right)\right].
\end{align*}
Hence
\begin{align*}
\sum_{k=1}^{b-1}k\left[\alpha_{k}-\alpha_{k+1}\right] & =b\left[\psi\left(1+\frac{z+b^{p}}{b^{l}}\right)-\psi\left(1+\frac{z}{b^{l}}\right)-\alpha_{b}\right]\\
 & =b\left[\psi\left(1+\frac{z+b^{p}}{b^{l}}\right)-\psi\left(1+\frac{z}{b^{l}}\right)\right]\\
 & -b\left[\psi\left(1+\frac{z+b^{p}}{b^{l+1}}\right)-\psi\left(1+\frac{z}{b^{l+1}}\right)\right].
\end{align*}
The right-hand side of Thm \ref{thm:Thm1-1} reads then
\[
\sum_{l=0}^{p-1}\frac{1}{b^{l}}\left[\psi\left(1+\frac{z+b^{p}}{b^{l}}\right)-\psi\left(1+\frac{z}{b^{l}}\right)-\psi\left(1+\frac{z+b^{p}}{b^{l+1}}\right)+\psi\left(1+\frac{z}{b^{l+1}}\right)\right]
\]
and can be easily simplified to the final result.

\end{proof}

\section{Asymptotics}

As a consequence of (\ref{eq:JN2p-1}), we deduce several expressions
for $J_{\infty}\left(x\right)$. Notice that a limit such as
\begin{equation}
\lim_{p\to+\infty}\sum_{n=1}^{2^{p}-1}s_{2}\left(n\right)\left(\frac{1}{\left(z+n\right)^{\alpha}}-\frac{1}{\left(z+n+1\right)^{\alpha}}\right)\label{eq:limit power}
\end{equation}
is the limit of a subsequence of the sequence 
\[
\sum_{n=1}^{q}s_{2}\left(n\right)\left(\frac{1}{\left(z+n\right)^{\alpha}}-\frac{1}{\left(z+n+1\right)^{\alpha}}\right).
\]
Since the latter is convergent as $q\to+\infty,$ any subsequence
will converge toward the same limit so that there is no ambiguity
in computing the limit (\ref{eq:limit power}).
\begin{thm}
\label{thm:Thm3}For any $\alpha>0$ such as $\alpha \ne 1$ and for any $z\ge0,$ the series
\[
\sum_{n=1}^{+\infty}s_{b}\left(n\right)\left(\frac{1}{\left(z+n\right)^{\alpha}}-\frac{1}{\left(z+n+1\right)^{\alpha}}\right)
\]
is equal to
\[
b\zeta\left(\alpha,1+z\right) + \left(1-b\right)\sum_{l=1}^{+\infty}\frac{1}{b^{l \alpha}}\zeta\left(\alpha,1+\frac{z}{b^l}\right).
%\sum_{k=1}^{b-1}k\left[\zeta\left(\alpha,\frac{k}{b}+\frac{z}{b^{l+1}}\right)-\zeta\left(\alpha,\frac{k+1}{b}+\frac{z}{b^{l+1}}\right)\right].
\]
%\[
%\sum_{l=0}^{+\infty}\frac{1}{\left(b^{l+1}\right)^{\alpha}}
%\sum_{k=1}^{b-1}k\left[\zeta\left(\alpha,\frac{k}{b}+\frac{z}{b^{l+1}}\right)-\zeta\left(\alpha,\frac{k+1}{b}+\frac{z}{b^{l+1}}\right)\right].
%\]
\end{thm}

\begin{proof}
We prove the case $b=2$ by showing that the term
\[
\sum_{l=0}^{p-1}\frac{1}{\left(2^{l+1}\right)^{\alpha}}\left[\zeta\left(\alpha,\frac{1}{2}+\frac{z+2^{p}}{2^{l+1}}\right)-\zeta\left(\alpha,1+\frac{z+2^{p}}{2^{l+1}}\right)\right]
\]
in \eqref{eq:zeta} converges to $0$ as $p\to+\infty.$ This sum can be expressed, using
the integral representation \cite[25.11.35]{NIST}
\[
\zeta\left(\alpha,\frac{a}{2}\right)-\zeta\left(\alpha,\frac{a+1}{2}\right)=\frac{2^{\alpha}}{\Gamma\left(\alpha\right)}\int_{0}^{+\infty}\frac{x^{\alpha-1}}{1+e^{-x}}e^{-ax}dx,
\]
as follows:
\[
\frac{1}{\Gamma\left(\alpha\right)}\int_{0}^{+\infty}\frac{x^{\alpha-1}}{1+e^{-x}}\sum_{l=0}^{p-1}2^{-\alpha l}e^{-\frac{1}{2^{l}}\left(z+2^{p}\right)x}dx.
\]
The sum can be upper-bounded as follows
\[
\sum_{l=0}^{p-1}2^{-\alpha l}e^{-\frac{1}{2^{l}}\left(z+2^{p}\right)x}\le\sum_{l=0}^{p-1}2^{-\alpha l}e^{-2^{p-l}x}.
\]
A study of the function
\[
h\left(u\right)=2^{-\alpha u}e^{-x \cdot 2^{p-u}}=e^{-\alpha u\log2-x2^{p-u}}
\]
shows that it satisfies
\[
h'\left(u\right)=h\left(u\right)\left(-\alpha\log2-x2^{p-u}\log2\right).
\]
It is positive on $u\in\left[0,p\right]$ and its unique maximum value
over this interval can be computed as follows:
\[
\max_{u\in\left[0,p\right]}h\left(u\right)=h\left(p-\frac{\log\left(\alpha/x\right)}{\log2}\right)=2^{\alpha\left(\frac{\log\left(\alpha/x\right)}{\log2}-p\right)}e^{-x \cdot 2^{\frac{\log\left(\alpha/x\right)}{\log2}}}\le K \cdot 2^{-\alpha p},
\]
where the positive constant $K$ depends on $x$ and $\alpha$ only.
Therefore, we have the bound
\[
\sum_{l=0}^{p-1}2^{-\alpha l}e^{-\frac{1}{2^{l}}\left(z+2^{p}\right)x}\le Kp2^{-\alpha p}.
\]
Since $\alpha>0,$ this bound goes to $0$ with $p,$ which proves
the result. The case of an arbitrary value of $b$ can be proved accordingly.
\end{proof}
\begin{cor}
The series
\begin{align}
J_{\infty}^{\left(b\right)}\left(x\right) & =\sum_{n=1}^{+\infty}\frac{s_{b}\left(n\right)}{\left(x+n\right)\left(x+n+1\right)}\label{eq:Jinfty(x)}
\end{align}
can be expressed as follows
%or equivalently as follows
%\begin{equation}
%%J_{\infty}^{\left(b\right)}\left(x\right)=\sum_{l\ge0}\frac{1}{2^{l}}\beta\left(\frac{x}{2^{l}}+1\right)
%J_{\infty}^{\left(b\right)}\left(x\right)= \frac{b}{b-1} \log b
%+ \sum_{l\ge0} \frac{1}{b^{l+1}} 
%\left[
%\left(b-1\right)\psi\left(1+\frac{x}{b^{l+1}}\right)+\psi\left(\frac{x}{b^{l+1}}\right)-b\psi\left(\frac{x}{b^l}\right)
%\right]
%\label{eq:Jinftybeta2}
%\end{equation}
\begin{equation}
%J_{\infty}^{\left(b\right)}\left(x\right)=\sum_{l\ge0}\frac{1}{2^{l}}\beta\left(\frac{x}{2^{l}}+1\right)
J_{\infty}^{\left(b\right)}\left(x\right)= \frac{b}{b-1} \log b
+ \sum_{l\ge0} \frac{1}{b^{l}} 
\left[
\psi\left(\frac{x}{b^{l+1}}\right)-\psi\left(\frac{x}{b^{l}}\right)+\frac{b-1}{x}
\right].
\label{eq:Jinftybeta3}
\end{equation}
It has a Taylor expansion at $x=0$ given by
\begin{align}
\label{eq:JinftyTaylor}
%J_{\infty}^{\left(b\right)}\left(x\right) & =2\log2-2\sum_{n\ge1}\frac{1-2^{n}}{1-2^{n+1}}\frac{\psi^{\left(n\right)}\left(1\right)}{n!}x^{n}\\
% & =2\log2+\sum_{n\ge1}\left(-1\right)^{n}\zeta\left(n+1\right)\left(\frac{2^{n+1}-2}{2^{n+1}-1}\right)x^{n}.
J_{\infty}^{\left(b\right)}\left(x\right)&=\frac{b}{b-1}\log b-b\sum_{n\ge1}\frac{1-b^{n}}{1-b^{n+1}}\frac{\psi^{\left(n\right)}\left(1\right)}{n!}x^{n}\\&=\frac{b}{b-1}\log b+\sum_{n\ge1}\left(-1\right)^{n}\zeta\left(n+1\right)\frac{b^{n+1}-b}{b^{n+1}-1}x^{n}.
\end{align}
\end{cor}

\begin{proof}
The identity 
\begin{equation}
%J_{\infty}^{\left(b\right)}\left(x\right)=\sum_{l\ge0}\frac{1}{2^{l}}\beta\left(\frac{x}{2^{l}}+1\right)
J_{\infty}^{\left(b\right)}\left(x\right)=\sum_{l\ge0}\frac{1}{b^{l}}\sum_{k=1}^{b-1}\frac{k}{b}\left[\psi\left(\frac{k+1}{b}+\frac{x}{b^{l+1}}\right)-\psi\left(\frac{k}{b}+\frac{x}{b^{l+1}}\right)\right].
\label{eq:Jinftybeta}
\end{equation}
is obtained either as the limit case $\alpha=1$
in Theorem \ref{thm:Thm3}, or as the limit case $p \to \infty$ followed by elementary simplifications in Thm \ref{thm:Thm1-1}.
Using the multiplication theorem for the digamma function in identity \eqref{eq:Jinftybeta} gives the second result. The Taylor expansion is a direct consequence of this identity. 
%Taking the limit $p\to+\infty$ in (\ref{eq:recurrence J2n-1}) gives
%\begin{equation}
%J_{\infty}\left(x\right)-\frac{1}{2}J_{\infty}\left(\frac{x}{2}\right)=\eta_{\infty}\left(x\right)\label{eq:limit recurrence}
%\end{equation}
%with
%\[
%\eta_{\infty}\left(x\right)=\beta\left(x+1\right),
%\]
%since
%\[
%\lim_{N\to+\infty}\beta\left(x+2N+1\right)=0.
%\]
%The Taylor expansion
%\[
%\psi\left(x+1\right)=\sum_{n\ge0}\frac{\psi^{\left(n\right)}\left(1\right)}{n!}x^{n}=-\gamma+\sum_{n\ge1}\left(-1\right)^{n}\zeta\left(n+1\right)z^{n}
%\]
%gives, for Stirling's beta function, 
%\begin{align*}
%\beta\left(x+1\right) & =\frac{1}{2}\psi\left(\frac{x}{2}+1\right)-\frac{1}{2}\psi\left(\frac{x}{2}+\frac{1}{2}\right)=\psi\left(\frac{x}{2}+1\right)-\psi\left(x+1\right)+\log2\\
% & =\log2+\sum_{n\ge0}\frac{\psi^{\left(n\right)}\left(1\right)}{n!}x^{n}\left[\frac{1}{2^{n}}-1\right].
%\end{align*}
%Substituting the Taylor expansion
%\[
%J_{\infty}\left(x\right)=\sum_{n\ge0}\frac{a_{n}}{n!}x^{n}
%\]
%in (\ref{eq:limit recurrence}) gives 
%\[
%\sum_{n\ge0}\frac{a_{n}}{n!}x^{n}\left[1-\frac{1}{2^{n+1}}\right]=\log2+\sum_{n\ge0}\frac{\psi^{\left(n\right)}\left(1\right)}{n!}x^{n}\left[\frac{1}{2^{n}}-1\right],
%\]
%so that
%\[
%a_{0}=2\log2.
%\]
%Similarly, for $n\ge1,$
%\[
%\frac{a_{n}}{n!}=\frac{\psi^{\left(n\right)}\left(1\right)}{n!}\frac{\left(\frac{1}{2}\right)^{n}-1}{1-\left(\frac{1}{2}\right)^{n+1}}=\left(-1\right)^{n+1}\zeta\left(n+1\right)\left(\frac{1}{1-\left(\frac{1}{2}\right)^{n+1}}-2\right).
%\]
% 
\end{proof}
\begin{rem}
The special case $x=0$ in (\ref{eq:JinftyTaylor}) coincides with identity \eqref{logb}.
\end{rem}

Integration of (\ref{eq:Jinftybeta}) yields the following infinite
product:

\begin{thm}
The following identity holds for all $z\in\mathbb{R},$
\begin{equation}
%\prod_{n\ge1}\left(\frac{1+\frac{z}{n}}{1+\frac{z}{n+1}}\right)^{s_{2}\left(n\right)}=\frac{2^{2z}}{\Gamma\left(z+1\right)}\prod_{l\ge1}\Gamma\left(\frac{z}{2^{l}}+1\right). 
\prod_{n\ge1}\left(\frac{1+\frac{z}{n}}{1+\frac{z}{n+1}}\right)^{s_{b}\left(n\right)}=
b^{z\frac{b}{b-1}}\prod_{l\ge0}\frac{\Gamma^{b}\left(1+\frac{z}{b^{l+1}}\right)}{\Gamma\left(1+\frac{z}{b^{l}}\right)}.
%\prod_{l\ge0}\prod_{k=1}^{b-1}\left[\frac{\Gamma\left(\frac{z}{b^{l+1}}+\frac{k+1}{b}\right)}{\Gamma\left(\frac{k+1}{b}\right)}\frac{\Gamma\left(\frac{k}{b}\right)}{\Gamma\left(\frac{z}{b^{l+1}}+\frac{k}{b}\right)}\right]^{k}.
\label{eq:infinite product}
\end{equation}
\end{thm}

\begin{proof}
Integrating both representations
\[
%J_{\infty}\left(x\right)=\sum_{n=1}^{+\infty}\frac{s_{2}\left(n\right)}{\left(x+n\right)\left(x+n+1\right)}=\sum_{l\ge0}\frac{1}{2^{l}}\beta\left(\frac{x}{2^{l}}+1\right)
J_{\infty}^{\left(b\right)}\left(x\right)=\sum_{n=1}^{+\infty}\frac{s_{b}\left(n\right)}{\left(x+n\right)\left(x+n+1\right)}=\sum_{l\ge0}\frac{1}{b^{l}}\sum_{k=1}^{b-1}\frac{k}{b}\left[\psi\left(\frac{k+1}{b}+\frac{x}{b^{l+1}}\right)-\psi\left(\frac{k}{b}+\frac{x}{b^{l+1}}\right)\right]
\]
gives on one side
\[
\int_{0}^{z}\sum_{n=1}^{+\infty}\frac{s_{b}\left(n\right)}{\left(x+n\right)\left(x+n+1\right)}dx=\log\prod_{n\ge1}\left(\frac{1+\frac{z}{n}}{1+\frac{z}{n+1}}\right)^{s_{b}\left(n\right)}
\]
and on the other side
\begin{align*}
%\int_{0}^{z}\sum_{l\ge0}\frac{1}{2^{l}}\beta\left(\frac{x}{2^{l}}+1\right)dx=\log\prod_{l\ge0}\frac{\Gamma\left(\frac{z}{2^{l+1}}+1\right)}{\frac{\Gamma\left(\frac{z}{2^{l+1}}+\frac{1}{2}\right)}{\Gamma\left(\frac{1}{2}\right)}}.
&\int_{0}^{z}\sum_{l\ge0}\frac{1}{b^{l}}\sum_{k=1}^{b-1}\frac{k}{b}\left[\psi\left(\frac{k+1}{b}+\frac{x}{b^{l+1}}\right)-\psi\left(\frac{k}{b}+\frac{x}{b^{l+1}}\right)\right]dx \\
&= \sum_{l\ge0}\frac{1}{b^{l}}\sum_{k=1}^{b-1}\frac{k}{b} b^{l+1}\log\left(\frac{\Gamma\left(\frac{z}{b^{l+1}}+\frac{k+1}{b}\right)}{\Gamma\left(+\frac{k+1}{b}\right)}\frac{\Gamma\left(\frac{k}{b}\right)}{\Gamma\left(\frac{z}{b^{l+1}}+\frac{k}{b}\right)}\right)\\
&=\log \prod_{l\ge0}\prod_{k=1}^{b-1}\left[\frac{\Gamma\left(\frac{z}{b^{l+1}}+\frac{k+1}{b}\right)}{\Gamma\left(\frac{k+1}{b}\right)}\frac{\Gamma\left(\frac{k}{b}\right)}{\Gamma\left(\frac{z}{b^{l+1}}+\frac{k}{b}\right)}\right]^{k}.
\end{align*}
Taking the exponential of both sides gives
\[
\prod_{n\ge1}\left(\frac{1+\frac{z}{n}}{1+\frac{z}{n+1}}\right)^{s_{b}\left(n\right)}=\prod_{l\ge0}\prod_{k=1}^{b-1}\left[\frac{\Gamma\left(\frac{z}{b^{l+1}}+\frac{k+1}{b}\right)}{\Gamma\left(\frac{k+1}{b}\right)}\frac{\Gamma\left(\frac{k}{b}\right)}{\Gamma\left(\frac{z}{b^{l+1}}+\frac{k}{b}\right)}\right]^{k}.
\]
Denote 
\[
\alpha_{k}=\frac{\Gamma\left(\frac{z}{b^{l+1}}+\frac{k}{b}\right)}{\Gamma\left(\frac{k}{b}\right)}
\]
so that the right-hand side reads
\begin{align*}
\prod_{l\ge0}\prod_{k=1}^{b-1}\left(\frac{\alpha_{k+1}}{\alpha_{k}}\right)^{k} & =\prod_{l\ge0}\frac{\alpha_{2}}{\alpha_{1}}\frac{\alpha_{3}^{2}}{\alpha_{2}^{2}}\cdots\left(\frac{\alpha_{b}}{\alpha_{b-1}}\right)^{b-1}\\
 & =\prod_{l\ge0}\prod_{k=1}^{b-1}\left(\frac{\alpha_{b}}{\alpha_{k}}\right)
\end{align*}
and can be expressed as
\[
\frac{\alpha_{b}}{\alpha_{k}}=\frac{\Gamma\left(\frac{z}{b^{l+1}}+1\right)}{\Gamma\left(1\right)}\frac{\Gamma\left(\frac{k}{b}\right)}{\Gamma\left(\frac{z}{b^{l+1}}+\frac{k}{b}\right)}=\frac{z}{b^{l+1}}B\left(\frac{z}{b^{l+1}},\frac{k}{b}\right)
\]
where $B$ is Euler's beta function. Thus right-hand side is
\[
\prod_{l\ge0}\left(\frac{z}{b^{l+1}}\right)^{b-1}\prod_{k=1}^{b-1}B\left(\frac{z}{b^{l+1}},\frac{k}{b}\right).
\]
Using Euler's duplication formula allows to simplify the product of Beta functions as 
\[
\prod_{k=1}^{b-1}B\left(\frac{z}{b^{l+1}},\frac{k}{b}\right)=b^{b\frac{z}{b^{l+1}}-1}\frac{\Gamma^{b}\left(\frac{z}{b^{l+1}}\right)}{\Gamma\left(b\frac{z}{b^{l+1}}\right)}
\]
so that the right-hand side is
\[
\prod_{l\ge0}\left(\frac{z}{b^{l+1}}\right)^{b-1}b^{\frac{z}{b^{l}}-1}\frac{\Gamma^{b}\left(\frac{z}{b^{l+1}}\right)}{\Gamma\left(\frac{z}{b^{l}}\right)}=\prod_{l\ge0}b^{\frac{z}{b^{l}}}\frac{\Gamma^{b}\left(1+\frac{z}{b^{l+1}}\right)}{\Gamma\left(1+\frac{z}{b^{l}}\right)}=b^{z\frac{b}{b-1}}\prod_{l\ge0}\frac{\Gamma^{b}\left(1+\frac{z}{b^{l+1}}\right)}{\Gamma\left(1+\frac{z}{b^{l}}\right)},
\]
which is the desired result.
%Using the duplication formula
%\[
%\prod_{l\ge0}\frac{\Gamma\left(\frac{z}{2^{l+1}}+\frac{1}{2}\right)}{\Gamma\left(\frac{1}{2}\right)}=\frac{\Gamma\left(z+1\right)}{2^{2z}}
%\]
%allows us to deduce the result.
\end{proof}
Using Legendre's duplication formula for the Gamma function and Knar's formula
\[
\Gamma\left(1+z\right)=2^{2z}\prod_{k\ge1}\frac{\Gamma\left(\frac{1}{2}+\frac{z}{2^{k}}\right)}{\Gamma\left(\frac{1}{2}\right)},
\]
the binary case $b=2$ simplifies as follows:
\begin{cor}
We have
\begin{equation}
\label{eq:infinite product base 2}
\prod_{n\ge1}\left(\frac{1+\frac{z}{n}}{1+\frac{z}{n+1}}\right)^{s_{2}\left(n\right)}=\frac{2^{2z}}{\Gamma\left(z+1\right)}\prod_{l\ge1}\Gamma\left(\frac{z}{2^{l}}+1\right). 
\end{equation}
\end{cor}
Considering this identity with special values of $z$ gives the following result.
\begin{cor}
We have %
\begin{comment}
Mathematica Checked
\end{comment}
\begin{equation}
\prod_{n\ge1}\left(\frac{1+\frac{1}{n}}{1+\frac{1}{n+1}}\frac{1+\frac{1}{2n+2}}{1+\frac{1}{2n}}\right)^{s_{2}\left(n\right)}=\frac{\pi}{2}.\label{eq:p=00003D0}
\end{equation}
\end{cor}

\begin{proof}
Evaluating (\ref{eq:infinite product base 2}) at $z=1$ gives %
\begin{comment}
Mathematica Checked
\end{comment}
\[
\prod_{n\ge1}\left(\frac{1+\frac{1}{n}}{1+\frac{1}{n+1}}\right)^{s_{2}\left(n\right)}=4\prod_{l\ge1}\Gamma\left(\frac{1}{2^{l}}+1\right),
\]
while at $z=\frac{1}{2}$ we obtain %
\begin{comment}
Mathematica Checked
\end{comment}
\[
\prod_{n\ge1}\left(\frac{1+\frac{1}{2n}}{1+\frac{1}{2n+2}}\right)^{s_{2}\left(n\right)}=\frac{4}{\sqrt{\pi}}\prod_{l\ge1}\Gamma\left(\frac{1}{2^{l+1}}+1\right).
\]
Dividing the former identity by the latter gives the result.
\end{proof}
More generally, evaluating (\ref{eq:infinite product}) at $z=\frac{1}{2^{p}}$
with $p$ a positive integer gives %
\begin{comment}
Mathematica Checked
\end{comment}
\[
\prod_{n\ge1}\left(\frac{1+\frac{1}{n \cdot 2^{p}}}{1+\frac{1}{\left(n+1\right) \cdot 2^{p}}}\right)^{s_{2}\left(n\right)}=\frac{2^{2^{1-p}}}{\Gamma\left(\frac{1}{2^{p}}+1\right)}\prod_{l\ge1}\Gamma\left(\frac{1}{2^{p+l}}+1\right),
\]
from which we deduce %
\begin{comment}
Mathematica Checked
\end{comment}
\begin{align*}
\prod_{n\ge1}\left(\frac{1+\frac{1}{n}}{1+\frac{1}{n+1}}\frac{1+\frac{1}{\left(n+1\right)2^{p}}}{1+\frac{1}{n \cdot 2^{p}}}\right)^{s_{2}\left(n\right)} & =2^{2-2^{1-p}}\Gamma\left(\frac{1}{2^{p}}+1\right)\prod_{1\le l\le p}\Gamma\left(\frac{1}{2^{l}}+1\right)
\end{align*}
and %
\begin{comment}
Mathematica Checked
\end{comment}
\[
\prod_{n\ge1}\left(\frac{1+\frac{1}{n \cdot 2^{p}}}{1+\frac{1}{\left(n+1\right)2^{p}}}\frac{1+\frac{1}{\left(n+1\right)2^{p+1}}}{1+\frac{1}{n \cdot 2^{p+1}}}\right)^{s_{2}\left(n\right)}=2^{2^{-p}}\frac{\Gamma^{2}\left(1+\frac{1}{2^{p+1}}\right)}{\Gamma\left(1+\frac{1}{2^{p}}\right)}.
\]
The case $p=0$ is (\ref{eq:p=00003D0}) and the case $p=1$ is %
\begin{comment}
Mathematica Checked
\end{comment}
\begin{align*}
\prod_{n\ge1}\left(\frac{1+\frac{1}{2n}}{1+\frac{1}{2n+2}}\frac{1+\frac{1}{4n+4}}{1+\frac{1}{4n}}\right)^{s_{2}\left(n\right)} & =2\sqrt{\frac{2}{\pi}}\Gamma^{2}\left(\frac{5}{4}\right)=\frac{\pi}{2}\prod_{k\ge1}\tanh^{2}\left(k\frac{\pi}{2}\right),
\end{align*}
where the infinite product representation is obtained as a consequence
of the infinite product representation
\[
\prod_{k\ge1}\tanh\left(k\frac{\pi}{2}\right)=\left(\frac{\pi}{2}\right)^{\frac{1}{4}}\frac{1}{\Gamma\left(\frac{3}{4}\right)}
\]
and the duplication formula which links $\Gamma\left(\frac{1}{4}\right),\thinspace\thinspace\Gamma\left(\frac{3}{4}\right)$
and $\Gamma\left(\frac{1}{2}\right).$ 

As a consequence of the identity \cite[Table 3, (iii)]{Borwein Zucker}
\[
\Gamma\left(\frac{1}{4}\right)=2\pi^{\frac{1}{4}}\sqrt{K\left(\frac{1}{\sqrt{2}}\right)},
%%=2\pi^{\frac{1}{4}}\sqrt{K\left(\frac{\theta_{2}^{2}\left(e^{-\pi}\right)}{\theta_{3}^{2}\left(e^{-\pi}\right)}\right)},
\]
where $K\left(k\right)$ is the complete elliptic integral of the
first kind
\[
K\left(k\right) = \int_{0}^{\frac{\pi}{2}}\frac{dt}{\sqrt{1-k^2 \sin^2 t}},
\]
we also have
\[
\prod_{n\ge1}\left(\frac{1+\frac{1}{2n}}{1+\frac{1}{2n+2}}\frac{1+\frac{1}{4n+4}}{1+\frac{1}{4n}}\right)^{s_{2}\left(n\right)}=\frac{1}{\sqrt{2}}K\left(\frac{1}{\sqrt{2}}\right).
\]
The case $p=2$ is
\[
\prod_{n\ge1}\left(\frac{1+\frac{1}{4n}}{1+\frac{1}{4n+4}}\frac{1+\frac{1}{8n+8}}{1+\frac{1}{8n}}\right)^{s_{2}\left(n\right)}=2^{\frac{1}{4}}\frac{\Gamma^{2}\left(1+\frac{1}{8}\right)}{\Gamma\left(1+\frac{1}{4}\right)}.
\]
\begin{rem}
Taking the $p-$th derivative at $x=0$ in (\ref{eq:Jinfty(x)})
gives
\[
J_{\infty}^{\left(p\right)}\left(0\right)=\left(-1\right)^{p}p!\sum_{n\ge1}s_{2}\left(n\right)\left[\frac{1}{n^{p+1}}-\frac{1}{\left(n+1\right)^{p+1}}\right],
\]
which is also equal, by (\ref{eq:JinftyTaylor}), to 
\[
J_{\infty}^{\left(p\right)}\left(0\right)=-2\frac{1-2^{p}}{1-2^{p+1}}\psi^{\left(p\right)}\left(1\right).
\]
Therefore, with
\[
\frac{\psi^{\left(p\right)}\left(1\right)}{p!}=\left(-1\right)^{p+1}\zeta\left(p+1\right),
\]
and after replacing $p+1$ with $p,$ we recover the well-known result 
\cite{Allouche Shallit-1}
\[
\sum_{n\ge1}s_{2}\left(n\right)\left[\frac{1}{n^{p}}-\frac{1}{\left(n+1\right)^{p}}\right]=\frac{2^{1-p}-1}{2^{-p}-1}\sum_{n\ge1}\frac{1}{n^{p}}.
\]
\end{rem}

\section{Lambert series}

\subsection{Inverse Laplace transform}

The inverse Laplace transform provides variations of the previous
identities. We use the inverse Laplace formulas
\[
\mathcal{L}^{-1}\left[\frac{1}{p+a}\right]=e^{-au},
\]
and \cite[Entry 3.1.4.2]{Prudnikov Laplace}
\[
\mathcal{L}^{-1}\left[\psi\left(ap+b\right)-\psi\left(ap+c\right)\right]=\frac{1}{a}\frac{e^{-\frac{c}{a}u}-e^{-\frac{b}{a}u}}{1-e^{-\frac{u}{a}}}.
\]
Taking the inverse Laplace transform of (\ref{eq:Jinftybeta}) gives
\[
%2\sinh\left(\frac{u}{2}\right)\sum_{n\ge1}s_{2}\left(n\right)e^{-\left(n+\frac{1}{2}\right)u}=\sum_{l\ge0}\frac{e^{-u \cdot 2^{l-1}}}{2\cosh\left(u \cdot 2^{l-1}\right)},\thinspace\thinspace u\ge1.
\sum_{n=1}^{+\infty}s_{b}\left(n\right)\left[e^{-nu}-e^{-\left(n+1\right)u}\right] = \sum_{l\ge0}\frac{1}{b^{l}}\sum_{k=1}^{b-1}\frac{k}{b}b^{l+1}\frac{e^{-kb^{l}u}-e^{-\left(k+1\right)b^{l}u}}{1-e^{-ub^{l+1}}}
.\]
Denoting $z=e^{-u}$ and simplifying gives the following result:
\begin{thm}
\label{thm:A-Lambert-series}A generating function for the sequence $s_{b}\left(n\right)$ is
\begin{comment}
Mathematica Checked
\end{comment}
\begin{align}
%\sum_{n\ge1}s_{2}\left(n\right)z^{n}=\frac{1}{1-z}\sum_{l\ge0}\frac{z^{2^{l}}}{z^{2^{l}}+1}.
\sum_{n\ge1}s_{b}\left(n\right)z^{n} & =\frac{1}{1-z}\sum_{l\ge0}\frac{z^{b^{l}}+2z^{2b^{l}}+\cdots+\left(b-1\right)z^{\left(b-1\right)b^{l}}}{1+z^{b^{l}}+z^{2b^{l}}+\cdots+z^{\left(b-1\right)b^{l}}}\\
 & =\frac{1}{1-z}\sum_{l\ge0}\frac{z^{b^{l}}-bz^{b^{l+1}}+\left(b-1\right)z^{\left(b+1\right)b^{l}}}{\left(1-z^{b^{l}}\right)\left(1-z^{b^{l+1}}\right)}. \nonumber
\end{align}
\end{thm}

%it appears as Pb. 41, Section 7.1.3 in \cite[page 187]{Knuth}
This result is not new \cite[Thm 1]{Adams watters}. The binary case is interesting since it involves a Lambert series:
\begin{equation}
\label{eq:Lambert infinite}
\sum_{n\ge1}s_{2}\left(n\right)z^{n}=\frac{1}{1-z}\sum_{l\ge0}\frac{z^{2^{l}}}{z^{2^{l}}+1}.
\end{equation}
%\begin{align*}
%\sum_{n\ge1}s_{b}\left(n\right)z^{n} & =\frac{1}{1-z}\sum_{l\ge0}\frac{z^{b^{l}}+2z^{2b^{l}}+\cdots+\left(b-1\right)z^{\left(b-1\right)b^{l}}}{1+z^{b^{l}}+z^{2b^{l}}+\cdots+z^{\left(b-1\right)b^{l}}}\\
% & =\frac{1}{1-z}\sum_{l\ge0}\frac{z^{b^{l}}-bz^{b^{l+1}}+\left(b-1\right)z^{\left(b+1\right)b^{l}}}{\left(1-z^{b^{l}}\right)\left(1-z^{b^{l+1}}\right)}.
%\end{align*}

Analogously, taking the inverse Laplace of the finite generating function
(\ref{eq:zeta}) gives a finite version of the previous result.
\begin{thm}
The finite generating function of the sequence $s_{2}\left(n\right)$
can be expressed as the finite Lambert series
\begin{equation}
%\sum_{n=1}^{2^{p}-1}s_{2}\left(n\right)z^{n}=\frac{1-z^{2^{p}}}{1-z}%\sum_{l=0}^{p-1}\frac{z^{2^{l}}}{1+z^{2^{l}}}.
\label{eq:Lambert finite}
\sum_{n\ge1}^{b^p-1}s_{b}\left(n\right)z^{n} 
%& =\frac{1}{1-z}\sum_{l\ge0}\frac{z^{b^{l}}+2z^{2b^{l}}+\cdots+\left(b-1\right)z^{\left(b-1\right)b^{l}}}{1+z^{b^{l}}+z^{2b^{l}}+\cdots+z^{\left(b-1\right)b^{l}}}\\
 =\frac{1-z^{b^{p}}}{1-z}\sum_{l\ge0}\frac{z^{b^{l}}-bz^{b^{l+1}}+\left(b-1\right)z^{\left(b+1\right)b^{l}}}{\left(1-z^{b^{l}}\right)\left(1-z^{b^{l+1}}\right)}.
\end{equation}
The binary case is
\[
\sum_{n\ge1}^{2^p-1}s_{2}\left(n\right)z^{n} 
=\frac{1-z^{2^{p}}}{1-z}\sum_{l=0}^{p-1}\frac{z^{2^{l}}}{1+z^{2^{l}}}.
\]
\end{thm}

Looking at the function
\[
\frac{1-z^{2^{p}}}{1-z}\sum_{l=0}^{p-1}\frac{z^{2^{l}}}{1+z^{2^{l}}}
\]
in the case $p=3$ produces
\begin{align*}
\frac{1-z^{8}}{1-z}\left[\frac{z}{1+z}+\frac{z^{2}}{1+z^{2}}+\frac{z^{4}}{1+z^{4}}\right] & =\frac{1-z^{8}}{1-z}\left[z\frac{1-z}{1-z^{2}}+z^{2}\frac{1-z^{2}}{1-z^{4}}+z^{4}\frac{1-z^{4}}{1-z^{8}}\right]\\
 & =z\left(1+z^{2}+z^{4}+z^{6}\right)\\
 & +z^{2}\left(1+z\right)\left(1+z^{4}\right)\\
 & +z^{4}\left(1+z+z^{2}+z^{3}\right).
\end{align*}

The first term $z\left(1+z^{2}+z^{4}+z^{6}\right)$ counts all the
occurrences of 1's as the rank $0$ (least significant) bit in the
binary representations of the numbers from $0$ to $7$. The second
term $z^{2}\left(1+z\right)\left(1+z^{4}\right)$ counts all the occurrences
of 1's as the rank $1$ bit in these numbers, and so on. Hence formula
(\ref{eq:Lambert finite}) has a simple combinatorial interpretation:
it counts the ones in the binary expansions of the integers between $0$ and $2^{p}-1$ {\it rank-wise}.  This interpretation extends directly to the case of an arbitrary base.

\subsection{The sequence $\Delta s_{2}\left(n\right)$}

The sequence 
\[
\Delta s_{2}\left(n\right)=s_{2}\left(n+1\right)-s_{2}\left(n\right)
\]
appears naturally in the generating function (\ref{eq:Lambert infinite}).
In this section we exhibit some links between this sequence and other
sequences that appear in number theory.

First, write (\ref{eq:Lambert infinite}) in the equivalent form
\[
\sum_{n\ge1}\Delta s_{2}\left(n-1\right)z^{n}=\sum_{l\ge0}\frac{z^{2^{l}}}{z^{2^{l}}+1}.
\]
Together with
\[
\sum_{n\ge1}s_{2}\left(n\right)\left[\frac{1}{n^{p}}-\frac{1}{\left(n+1\right)^{p}}\right]=\frac{2^{1-p}-1}{2^{-p}-1}\zeta\left(p\right),
\]
both identities are recognized as instances of a more general equivalence
that has appeared before \cite[Pb 11 page 300]{Borwein}.
First, denote the Dirichlet eta function by
\begin{equation}
\eta\left(s\right)=\sum_{n\ge1}\frac{\left(-1\right)^{n-1}}{n^{s}}=\left(1-2^{1-s}\right)\zeta\left(s\right).
\label{eq:eta}
\end{equation}
Given two sequences $\left\{ a_{n}\right\} $ and $\left\{ b_{n}\right\} $, 
\[
\eta\left(s\right)\sum_{n\ge1}\frac{a_{n}}{n^{s}}=\sum_{n\ge1}\frac{b_{n}}{n^{s}}
\]
if and only if
\[
\sum_{n\ge1}a_{n}\frac{x^{n}}{1+x^{n}}=\sum_{n\ge1}b_{n}x^{n}.
\]
The special case obtained here corresponds to $b_{n}=s_{2}\left(n\right)-s_{2}\left(n-1\right)$
and $a_{n}=\begin{cases}
1, & n=2^{p};\\
0, & \text{otherwise.}
\end{cases}$ This is a consequence of the way both Lambert and Dirichlet series encode \textit{divisors sums}, a viewpoint which we now will explore. 

Let us denote the $2-$adic valuation of $n,$ i.e.
the exponent of the highest power of $2$ that divides $n,$ by $\nu_{2}\left(n\right).$ 
We have the following result due to Legendre.
\begin{thm}
\label{thm:2-adic}
The $2-$adic valuation sequence and the $\Delta s_{2}\left(n-1\right)$
sequence are related as follows:
\begin{equation}
\Delta s_{2}\left(n-1\right)+\nu_{2}\left(n\right)=1,\thinspace\thinspace n\ge1.\label{eq:2-adic}
\end{equation}
As a consequence, 
\begin{equation}
s_{2}\left(n\right) = n- \sum_{k=1}^{n} \nu_{2}\left(k\right),\thinspace\thinspace n\ge1.\label{eq:2-adic-2}
\end{equation}
\end{thm}

Although this result is well-known, we propose a proof that is a consequence of the Lambert series representation \eqref{eq:Lambert infinite} obtained for the generating function of the sequence $s_2\left(n\right).$

\begin{proof}
Let $\left\{ \delta_{n}^{\left(2\right)}\right\}$ denote the indicator
of a power of 2 sequence defined by
\begin{equation}
\delta_{n}^{\left(2\right)}=\begin{cases}
1, & \text{ if $n=2^{l}$}; \\
0, & \text{ otherwise.}
\end{cases}\label{eq:a_n}
\end{equation}
The Lambert generating function (\ref{eq:Lambert infinite}) can be
expressed equivalently as follows:
\[
\sum_{n\ge1}\Delta s_{2}\left(n-1\right)z^{n}=\sum_{n\ge1}\delta_{n}^{\left(2\right)}\frac{z^{n}}{1+z^{n}}.
\]
Expansion of the right-hand side gives
\[
\sum_{n\ge1}\delta_{n}^{\left(2\right)}\frac{z^{n}}{1+z^{n}}=\sum_{n\ge1}\delta_{n}^{\left(2\right)}\sum_{j\ge1}z^{nj}\left(-1\right)^{j+1}=\sum_{j\ge1}z^{j}\sum_{d\vert j}\delta_{d}^{\left(2\right)}\left(-1\right)^{\frac{j}{d}+1},
\]
so that
\begin{equation}
\Delta s_{2}\left(n-1\right)=\sum_{d\vert n}\left(-1\right)^{\frac{n}{d}+1}\delta_{d}^{\left(2\right)}.\label{eq:Deltas2(n-1)}
\end{equation}
Since $\left\{ \delta_{n}^{\left(2\right)}\right\} $ is the indicator
sequence for the powers of $2,$ only those divisors $d$ that are
powers of $2$ contribute to this sum. More precisely, each of the
$\nu_{2}\left(n\right)$ values of $d\in\left\{ 1,2,\ldots,2^{\nu_{2}\left(n\right)-1}\right\} $
contributes a $\left(-1\right)$ term since, for all these values,
$\frac{n}{d}$ is a non-zero power of $2$ so that $\frac{n}{d}+1$
is odd. Only the highest value $d=2^{\nu_{2}\left(n\right)}$ contributes
a $(+1)$ term, so that the sum of all these contributions is $1-\nu_{2}\left(n\right)$.
Identity \eqref{eq:2-adic-2} is obtained by replacing $n$ by $k$ in \eqref{eq:2-adic} and summing over $k$ in the interval $\left[1,n\right].$
\end{proof}
We thank V. H. Moll for suggesting that this result is also in fact equivalent
to Legendre's formula
\begin{equation}
\nu_{2}\left(n!\right)+s_{2}\left(n\right)=n.\label{eq:Legendre}
\end{equation}

Indeed,
\[
\nu_{2}\left(n\right)=\nu_{2}\left(n!\right)-\nu_{2}\left(\left(n-1\right)!\right)
\]
by the multiplicative property of the valuation function; then
using (\ref{eq:Legendre}) gives (\ref{eq:2-adic}), and proving the other
way is equally simple.

The next result is an equivalent form of (\ref{eq:Deltas2(n-1)}).
\begin{thm}
The sequence $\left\{ \delta_{n}^{\left(2\right)}\right\} $ is related
to the sequence $\left\{ \Delta s_{2}\left(n\right)\right\} $ by

\[
\delta_{n}^{\left(2\right)}=\sum_{d\vert n}c_{\frac{n}{d}}\Delta s_{2}\left(d-1\right),
\]
where the sequence $\left\{ c_{n}\right\} $ is defined in terms of
the M\"obius function $\left\{ \mu\left(n\right)\right\} $ by
\begin{equation}
c_{n}=\begin{cases}
\mu\left(n\right), & n\thinspace\thinspace\text{odd};\\
2^{\nu_{2}\left(n\right)-1}\mu\left(\frac{n}{2^{\nu_{2}\left(n\right)}}\right), & n\thinspace\thinspace\text{even.}
\end{cases}\label{eq:cn}
\end{equation}
\end{thm}

\begin{proof}
This result can be obtained by M\"obius-inverting identity (\ref{eq:Deltas2(n-1)}).
However, we give a second proof by rewriting the generating function
(\ref{eq:Lambert infinite}) in the form
\begin{equation}
\sum_{n\ge1}\frac{\delta_{n}^{\left(2\right)}}{n^{s}}=\frac{1}{\eta\left(s\right)}\sum\frac{\Delta s_{2}\left(n-1\right)}{n^{s}}\label{eq:Lambert 2}
\end{equation}
where, as above, $\left\{ \delta_{n}^{\left(2\right)}\right\} $ is
the indicator function (\ref{eq:a_n}). The inverse of Dirichlet's
eta function can be computed from (\ref{eq:eta}) so that 
\[
\frac{1}{\eta\left(s\right)}=\frac{1}{1-2^{1-s}}\frac{1}{\zeta\left(s\right)}=\frac{1}{1-2^{1-s}}\sum_{n\ge1}\frac{\mu\left(n\right)}{n^{s}}
\]
where $\mu\left(n\right)$ is the M\"obius function. Identifying gives
\[
\frac{1}{\eta\left(s\right)}=\sum_{n\ge1}\frac{c_{n}}{n^{s}}
\]
with $\left\{ c_{n}\right\} $ as in (\ref{eq:cn}). We deduce the desired result from
(\ref{eq:Lambert 2}).
\end{proof}
Our last result in this section is a combinatorial interpretation
of the sequence $\left\{ \delta_{n}^{\left(2\right)}\right\} .$
\begin{thm}
Denote the number
of partitions of $n$ into an even and odd number of parts by $p_{o}\left(n\right)$ and $p_{e}\left(n\right)$ respectively,
and  the number of powers of $2$ in all partitions
of $n$ into distinct parts by $P_{2}\left(n\right)$. 
The sequences $\left\{ \delta_{n}^{\left(2\right)}\right\} ,\thinspace$
$\left\{ P_{2}\left(n\right)\right\} \thinspace,$ $p_{o}\left(n\right)$
and $p_{e}\left(n\right)$ are related by the convolution formula
\begin{align*}
\delta_{n}^{\left(2\right)} & =P_{2}*\left(p_{e}-p_{o}\right)\\
 & =\sum_{k=1}^{n}P_{2}\left(k\right)\left(p_{e}\left(n-k\right)-p_{o}\left(n-k\right)\right).
\end{align*}
\end{thm}

\begin{proof}
Denote the $q-$Pochhammer symbol by
\[
\left(a;q\right)_{\infty}=\prod_{k\ge0}\left(1-aq^{k}\right).
\]
Merca and Schmidt
have recently shown \cite{Schmidt} the Lambert series factorization
\[
\sum_{n\ge1}a_{n}\frac{q^{n}}{1+q^{n}}=\frac{1}{\left(-q;q\right)_{\infty}}\sum_{n\ge1}q^{n}\sum_{k=1}^{n}\left(s_{o}\left(n,k\right)+s_{e}\left(n,k\right)\right)a_{k},
\]
where $s_{o}\left(n,k\right)$ and $s_{e}\left(n,k\right)$ are the
number of $k's$ in all partitions of $n$ into an odd and even number
of distinct parts respectively. We first notice that $\frac{1}{\left(-q;q\right)_{\infty}}$
is the generating function for the difference of the number of partitions
of $n$ into an even and an odd number of parts, so that 
\[
\frac{1}{\left(-q;q\right)_{\infty}}=\sum_{n\ge0}q^{n}\left(p_{e}\left(n\right)-p_{o}\left(n\right)\right).
\]
Moreover, $\left\{ \delta_{n}^{\left(2\right)}\right\} $ being the
indicator function for powers of two, we have
\[
\sum_{n\ge1}q^{n}\sum_{k=1}^{n}\left(s_{o}\left(n,k\right)+s_{e}\left(n,k\right)\right)a_{k}=\sum_{n\ge1}q^{n}P_{2}\left(n\right),
\]
where $P_{2}\left(n\right)$ counts the number of powers of $2$ in
all partitions of $n$ into distinct parts, so that 
\[
\sum_{n\ge1}\delta_{n}^{\left(2\right)}\frac{q^{n}}{1+q^{n}}=\sum_{n\ge0}q^{n}\left(p_{e}\left(n\right)-p_{o}\left(n\right)\right)\sum_{n\ge1}q^{n}P_{2}\left(n\right).
\]
Comparing coefficients gives the result.
\end{proof}

\section{About sums of the form $S_{N}\left(x\right)=\sum_{n=0}^{2^{N}-1}\left(-1\right)^{s_{2}\left(n\right)}f\left(x+n\right)$ }

\subsection{A general identity}

Here we give the first values of $S_{N}\left(x\right)$: 

with $N=1,$

\[
S_{1}\left(x\right)=f\left(x\right)-f\left(x+1\right)=-\Delta_{1}f\left(x\right),
\]
$N=2,$
\begin{align*}
S_{2}\left(x\right) & =f\left(x\right)-f\left(x+1\right)-f\left(x+2\right)+f\left(x+3\right)\\
 & =-\Delta_{1}f\left(x\right)+\Delta_{1}f\left(x+2\right)=\Delta_{1}\Delta_{2}f\left(x\right),
\end{align*}
and $N=3,$
\begin{align*}
S_{3}\left(x\right) & =f\left(x\right)-f\left(x+1\right)-f\left(x+2\right)+f\left(x+3\right)\\
 & -f\left(x+4\right)+f\left(x+5\right)+f\left(x+6\right)-f\left(x+7\right)\\
 & =S_{2}\left(x\right)-S_{2}\left(x+4\right)=-\Delta_{1}\Delta_{2}\Delta_{4}f\left(x\right).
\end{align*}
Here $\Delta_{i}$ is the finite forward difference operator
\[
\Delta_{i}f\left(x\right)=f\left(x+i\right)-f\left(x\right).
\]
These first cases suggest the general identity
\begin{equation}
\label{eq:Delta}
S_{N}\left(x\right)=\left(-1\right)^{N}\Delta_{1}\Delta_{2}\Delta_{4}\ldots\Delta_{2^{N-1}}f\left(x\right),
\end{equation}
which can be easily proved by induction on $N.$
In this formula, the operators $\Delta_{k}$ commute, and the right-hand
side can be expanded using
\begin{equation}
\Delta_{k}f\left(x\right)=f\left(x+k\right)-f\left(x\right)=kf'\left(x+kU\right)\label{eq:Delta}
\end{equation}
with the symbolic notation
\begin{equation}
\label{f(x+u)}
f\left(x+U\right)=\int_{0}^{1}f\left(x+u\right)du.
\end{equation}
We can then rewrite
\[
S_{N}\left(x\right)=\left(-1\right)^{N}2^{\frac{\left(N-1\right)N}{2}}f^{\left(N\right)}\left(x+U_{0}+2U_{1}+\cdots+2^{N-1}U_{N-1}\right).
\]
This formula can be simplified since, for example,
\begin{align*}
f\left(x+2U_{1}\right) & =\int_{0}^{1}f\left(x+2u\right)du=\frac{1}{2}\int_{0}^{2}f\left(x+v\right)dv\\
 & =\frac{1}{2}\int_{0}^{1}f\left(x+v\right)dv+\frac{1}{2}\int_{1}^{2}f\left(x+v\right)dv\\
 & =\frac{1}{2}f\left(x+U\right)+\frac{1}{2}f\left(x+U+1\right).
\end{align*}
More generally
\[
f\left(x+2^{i}U_{i}\right)=\frac{1}{2^{i}}\sum_{k_{i}=0}^{2^{i}-1}f\left(x+U_{i}+k_{i}\right),
\]
so that, with\footnote{Note that the symbol $V_N$ is defined by its action on $f$ as follows 
\[
f\left(x+V_N\right) = \int_{0}^{1}\ldots \int_{0}^{1} f\left(x+u_{0}+\cdots+u_{N-1}\right)du_{0}\ldots du_{N-1}
\]} $V_{N}=U_{0}+\cdots+U_{N-1},$
\begin{equation}
S_{N}\left(x\right)=\left(-1\right)^{N}\sum_{k_{1}=0}^{1}\ldots\sum_{k_{N-1}=0}^{2^{N-1}-1}f^{\left(N\right)}\left(x+V_{N}+k_{1}+\cdots+k_{N-1}\right).\label{eq:SN(x) 1}
\end{equation}
Now consider an arbitrary function $g$ and the multiple sums
\[
T_{N}\left(x\right)=\sum_{k_{1}=0}^{1}\ldots\sum_{k_{N}=0}^{2^{N}-1}g\left(x+k_{1}+\cdots+k_{N}\right).
\]
For example,
\[
T_{1}\left(x\right)=\sum_{k_{1}=0}^{1}g\left(x+k_{1}\right)=g\left(x\right)+g\left(x+1\right),
\]
\[
T_{2}\left(x\right)=\sum_{k_{1}=0}^{1}\sum_{k_{2}=0}^{3}g\left(x+k_{1}+k_{2}\right)=g\left(x\right)+2g\left(x+1\right)+2g\left(x+2\right)+2g\left(x+3\right)+g\left(x+4\right),
\]
and 
\begin{align*}
T_{3}\left(x\right) & =\sum_{k_{1}=0}^{1}\sum_{k_{2}=0}^{3}\sum_{k_{3}=0}^{7}g\left(x+k_{1}+k_{2}+k_{3}\right)=g\left(x\right)+3g\left(x+1\right)+5g\left(x+2\right)\\
 & +7g\left(x+3\right)+8g\left(x+4\right)+8g\left(x+5\right)+8g\left(x+6\right)+8g\left(x+7\right)+7g\left(x+8\right)\\
 & +5g\left(x+9\right)+3g\left(x+10\right)+g\left(x+11\right).
\end{align*}
The sequences of $2^{N+1}-N-2$ coefficients $\alpha_{k}^{\left(N\right)}$
that appear in the sum
\begin{equation}
T_{N}\left(x\right)=\sum_{k=0}^{2^{N}-N-1}\alpha_{k}^{\left(N\right)}g\left(x+k\right)\label{eq:alpha}
\end{equation}
are, for $N=1$ to $3$ and $0 \le k \le 2^{N+1}-N-2,$
\[
1,1
\]
\[
1,2,2,1
\]
\[
1,3,5,7,8,8,8,8,7,5,3,1
\]
%\begin{align*}
%\alpha_{0}^{\left(1\right)} & =\alpha_{1}^{\left(1\right)}=1
%\end{align*}
%\[
%\alpha_{0}^{\left(2\right)}=\alpha_{4}^{\left(2\right)}=1,\alpha_{1}^{\left(2\right)}=\alpha_{2}^{\left(2\right)}=\alpha_{3}^{\left(2\right)}=2,
%\]
%\[
%\alpha_{0}^{\left(3\right)}=\alpha_{11}^{\left(3\right)}=1,\alpha_{1}^{\left(3\right)}=\alpha_{10}^{\left(3\right)}=3,\alpha_{2}^{\left(3\right)}=\alpha_{9}^{\left(3\right)}=5
%\]
%\[
%\alpha_{3}^{\left(3\right)}=\alpha_{8}^{\left(3\right)}=7,\alpha_{4}^{\left(3\right)}=\alpha_{5}^{\left(3\right)}=\alpha_{6}^{\left(3\right)}=\alpha_{7}^{\left(3\right)}=8.
%\]
The sequence
\[
1,1,1,2,2,2,1,1,3,5,7,8,8,8,8,7,5,3,1,\dots
\]
appears as sequence A131823 in the Online Encyclopedia of Integer Sequences (OEIS), along with the generating function
%is given for the $2^{N+1}-N-1$ terms $\alpha_{k}^{\left(N\right)}$
\begin{equation}
\sum_{k=0}^{2^{N+1}-N-2}\alpha_{k}^{\left(N\right)}x^{k}=\prod_{i=0}^{N-1}\left(1+x^{2^{i}}\right)^{N-i}.\label{eq:gf}
\end{equation}
For example, the case $N=2$ is
\[
\left(1+x\right)^{2}\left(1+x^{2}\right)^{1}=\left(1+2x+x^{2}\right)\left(1+x^{2}\right)=1+2x+2x^{2}+2x^{3}+x^{4}.
\]
Although this result is not mentioned in its OEIS entry, this sequence has
a combinatorial interpretation as a restricted partition function,
namely
\[
\alpha_{k}^{\left(N\right)}=\#\left\{ 0\le k_{1}\le1,\ldots,0\le k_{N}\le2^{N}-1\thinspace\vert\thinspace k_{1}+\cdots+k_{N}=k\right\} .
\]
Also, it can be checked by induction on $N$ that
\begin{equation}
f^{\left(N\right)}\left(x+V_{N}\right)=\sum_{l=0}^{N}\binom{N}{l}\left(-1\right)^{N-l}f\left(x+l\right)=\Delta^{N}f\left(x\right).\label{eq:VN}
\end{equation}
For example, using \eqref{f(x+u)},
\[
f'\left(x+V_{1}\right)=f'\left(x+U_{0}\right)=f\left(x+1\right)-f\left(x\right)
\]
and
\begin{align*}
f''\left(x+V_{2}\right)=f''\left(x+U_{0}+U_{1}\right) & =f'\left(x+1+U_{1}\right)-f'\left(x+U_{1}\right)\\
 & =f\left(x+2\right)-2f\left(x+1\right)+f\left(x\right).
\end{align*}
Using (\ref{eq:VN}) and (\ref{eq:alpha}), we finally obtain the following result:
\begin{thm}
\label{thm:Thm1}For integer $N$ and an arbitrary function $f$ 
\begin{equation}
\sum_{n=0}^{2^{N}-1}\left(-1\right)^{s_{2}\left(n\right)}f\left(x+n\right)=\left(-1\right)^{N}\thinspace\sum_{k=0}^{2^{N}-N-1}\alpha_{k}^{\left(N-1\right)}\Delta^{N}f\left(x+k\right).
\label{eq:Sn Thm}
\end{equation}
\end{thm}

\begin{rem}
Note that, by substituting $x=1$ in (\ref{eq:gf}), we have the sum 
\[
\sum_{k=0}^{2^{N}-N-1}\alpha_{k}^{\left(N-1\right)}=2^{\frac{N\left(N-1\right)}{2}}
\]
so that the $2^N-N-1$ positive coefficients
\begin{equation}
p_{k}^{\left(N-1\right)}=\frac{1}{2^{\frac{N\left(N-1\right)}{2}}}\alpha_{k}^{\left(N-1\right)}\label{eq:pkN}
\end{equation}
sum to $1$ and can be interpreted as probability weights; we can
then rewrite identity (\ref{eq:Sn Thm}) as follows:
\begin{equation}
\sum_{n=0}^{2^{N}-1}\left(-1\right)^{s_{2}\left(n\right)}f\left(x+n\right)=\left(-1\right)^{N}2^{\frac{N\left(N-1\right)}{2}}\thinspace\sum_{k=0}^{2^{N}-N-1}p_{k}^{\left(N-1\right)}\Delta^{N}f\left(x+k\right).\label{eq:Sn Thm 2}
\end{equation}
\end{rem}

\begin{example}
Assuming that $f$ is a polynomial of degree $\le N-1,$ we have $\Delta^{N}f=0$
so that the right-hand side in (\ref{eq:Sn Thm}) equals zero. We then recover the fact that the Prouhet-Thue-Morse sequence $\left(-1\right)^{s_2\left(n\right)}$ %the theorem by Prouhet,
satisfies the following property \cite[Section 5.1]{Allouche ubiquitous}:
for all $x,$
\[
\sum_{n=0}^{2^{N}-1}\left(-1\right)^{s_{2}\left(n\right)}P\left(x+n\right)=0
\]
for any polynomial $P$ with $\deg P\le N-1.$
\end{example}

\begin{example}
The following identity, for $k\ge1,$
\begin{equation}
\sum_{n=0}^{2^{N}-1}\left(-1\right)^{s_{2}\left(n\right)}\left(x+n\right)^{N}=\left(-1\right)^{N}2^{\frac{N\left(N-1\right)}{2}}N!\label{eq:AS1}
\end{equation}
appears in Allouche and Shallit's book concerning automatic sequences \cite[page 116]{Allouche Shallit book}. This can be verified
using (\ref{eq:Sn Thm}) by choosing $f\left(x\right)=x^{N}$
so that $\Delta^{N}f\left(x\right)=N!$. This hinges on the fact that through (\ref{eq:Delta}),
\[
\Delta x^{N}=N\left(x+U\right)^{N-1}
\]
and more generally
\[
\Delta^{k}x^{N}=\frac{N!}{\left(N-k\right)!}\left(x+U_{1}+\cdots+U_{k}\right)^{N-k}.
\]
Therefore,
\[
\Delta^{N}x^{N}=N!
\]
and by noticing that the right-hand side of (\ref{eq:Sn Thm 2})
reduces to
\[
\left(-1\right)^{N}2^{\frac{N\left(N-1\right)}{2}}N!\thinspace\sum_{k=0}^{2^{N}-N-1}p_{k}^{\left(N-1\right)}
\]
with the sum equal to $1$, we recover (\ref{eq:AS1}).
\end{example}
\begin{example}
Still in Allouche and Shallit's book on automatic sequences \cite[Pb 39 page 115]{Allouche Shallit book}, one can find
\begin{equation}
\sum_{n=0}^{2^{N}-1}\left(-1\right)^{s_{2}\left(n\right)}\left(x+n\right)^{N+1}=\left(-1\right)^{N}\left(N+1\right)!2^{\frac{N\left(N-1\right)}{2}}\left(x+\frac{2^{N}-1}{2}\right).\label{eq:AS2}
\end{equation}
This can be verified by computing 
\[
\Delta^{N}x^{N+1}=\frac{\left(N+1\right)!}{1!}\left(x+U_{1}+\cdots+U_{N}\right)=\left(N+1\right)!\left(x+\frac{N}{2}\right);
\]
next we obtain for the right-hand side of (\ref{eq:Sn Thm 2})
\begin{equation}
\left(-1\right)^{N}\left(N+1\right)!2^{\frac{N\left(N-1\right)}{2}}\thinspace\sum_{k=0}^{2^{N}-N-1}p_{k}^{\left(N-1\right)}\left(x+k+\frac{N}{2}\right)=\left(-1\right)^{N}\left(N+1\right)!\thinspace\left[x+\frac{N}{2}+\sum_{k=0}^{2^{N}-N-1}kp_{k}^{\left(N-1\right)}\right].\label{eq:moment}
\end{equation}
The moment of order $1$, $\sum_{k=0}^{2^{N}-N-1}kp_{k}^{\left(N-1\right)}$,
can be deduced from the generating function (\ref{eq:gf}) by computing
its logarithmic derivative
\[
\frac{\frac{d}{dx}\prod_{i=0}^{N-1}\left(1+x^{2^{i}}\right)^{N-i}}{\prod_{i=0}^{N-1}\left(1+x^{2^{i}}\right)^{N-i}}=\sum_{i=0}^{N-1}\frac{\left(N-i\right)\left(1+x^{2^{i}}\right)^{N-i-2}2^{i}x^{2^{i}-1}}{\left(1+x^{2^{i}}\right)^{N-i}}
\]
and evaluating at $x=1$ to obtain
\[
\sum_{k=0}^{2^{N}-N-1}kp_{k}^{\left(N-1\right)}=\frac{1}{2}2^{\frac{N\left(N-1\right)}{2}}\left(2^{N}-N-1\right).
\]
Inserting in (\ref{eq:moment}) gives (\ref{eq:AS2}).

We notice that two other proofs of \eqref{eq:AS2} appear in \cite{Bateman}, one by D. Callan based on purely combinatorial arguments, and the other by R. Stong that uses the representation \eqref{eq:Delta}.
\end{example}
%\begin{example}
%The Bernoulli Barnes polynomials $B_{m,N}\left(x;a_{1},\dots,a_{N}\right)$
%are defined by the generating functon
%\[
%\sum_{m\ge0}\frac{B_{m,N}\left(x;a_{1},\dots,a_{N}\right)}{m!}z^{m}=e^{zx}\prod_{k=1}^{N}\frac{a_{k}z}{e^{a_{k}z}-1}.
%\]
%Take the Bernoulli Barnes polynomial
%\[
%f\left(x\right)=B_{m,N}\left(x;1,2,4,\dots,2^{N-1}\right)=\left(x+B_{0}+2B_{1}+\dots+2^{N-1}B_{N-1}\right)^{m}
%\]
%with degree $m>N$ so that, using (\ref{eq:SN(x) 1}),
%\begin{align*}
%\sum_{n=0}^{2^{N}-1}\left(-1\right)^{s_{2}\left(n\right)}B_{m,N}\left(x+n;1,2,\dots,2^{N-1}\right) & =\left(-1\right)^{N}\frac{d^{N}}{dx^{N}}B_{m,N}\left(x+U_{0}+2U_{1}+\dots+2^{N-1}U_{N-1}\right)\\
% & =\left(-1\right)^{N}\frac{d^{N}}{dx^{N}}\left(x+U_{0}+\dots+2^{N-1}U_{N-1}+B_{0}+\dots+2^{N-1}B_{N-1}\right)^{m}\\
% & =\left(-1\right)^{N}\frac{m!}{\left(m-N\right)!}\left(x+U_{0}+\dots+2^{N-1}U_{N-1}+B_{0}+\dots+2^{N-1}B_{N-1}\right)^{m-N}\\
% & =\left(-1\right)^{N}\frac{m!}{\left(m-N\right)!}x^{m-N}.
%\end{align*}
%If the degree $m<N$ then the sum $S_{N}\left(x\right)$ is equal
%to $0;$ if $m=N,$ it is equal to $\left(-1\right)^{N}N!$
%\end{example}

\section{A general formula}\label{section6}
In this last section, we derive a general formula for sums of the form $\sum_{n} s_{2}\left(n\right)f\left(n\right).$
\subsection{The base $b=2$ case}\label{recursivesec}

Consider an arbitrary sequence $f\left(n\right)$ and compute, as
usual, 
\begin{align*}
\sum_{n\ge1}s_{2}\left(n\right)f\left(n\right) & =\sum_{n\ge1}s_{2}\left(2n\right)f\left(2n\right)+\sum_{n\ge0}s_{2}\left(2n+1\right)f\left(2n+1\right)\\
 & =\sum_{n\ge1}s_{2}\left(n\right)f\left(2n\right)+\sum_{n\ge0}s_{2}\left(n\right)f\left(2n+1\right)+\sum_{n\ge0}f\left(2n+1\right),
\end{align*}
so that
\[
\sum_{n\ge0}s_{2}\left(n\right)\left[f\left(n\right)-f\left(2n\right)-f\left(2n+1\right)\right]=\sum_{n\ge0}f\left(2n+1\right).
\]
Hence, given an arbitrary function $g\left(n\right),$ solving for
$f\left(n\right)$ in the implicit equation
\begin{equation}
g\left(n\right)=f\left(n\right)-f\left(2n\right)-f\left(2n+1\right)\label{eq:g(n)}
\end{equation}
would allow for the computation of the sum $\sum_{n\ge0}s_{2}\left(n\right)g\left(n\right)$
as follows:
\[
\sum_{n\ge0}s_{2}\left(n\right)g\left(n\right)=\sum_{n\ge0}f\left(2n+1\right).
\]
The next theorem gives the solution to equation (\ref{eq:g(n)}).
\begin{thm}
\label{thm:Sum base 2}The equation
\begin{equation}
g\left(n\right)=f\left(n\right)-f\left(2n\right)-f\left(2n+1\right)\label{eq:equation}
\end{equation}
has formal solution
\begin{equation}
f\left(n\right)=\sum_{k\ge0}\sum_{l=0}^{2^{k}-1}g\left(2^{k}n+l\right).\label{eq:f(n)}
\end{equation}
As a consequence,
\[
\sum_{n\ge1}s_{2}\left(n\right)g\left(n\right)=\sum_{n\ge0}f\left(2n+1\right)=\sum_{n,k\ge0}\sum_{l=0}^{2^{k}-1}g\left(2^{k+1}n+2^{k}+l\right).
\]
\end{thm}

\begin{proof}
We introduce the dilation operator 
\[
\eta f\left(n\right)=f\left(2n\right)
\]
and the shift operator
\[
\delta f\left(n\right)=f\left(n+\frac{1}{2}\right).
\]
We then rewrite equation (\ref{eq:equation}) symbolically as follows \footnote{We borrow this elegant operational technique from G. E. Andrews \cite{Andrews}.}
\[
g\left(n\right)=\left(1-\eta-\eta\delta\right)f\left(n\right),
\]
so that we formally have
\[
f\left(n\right)=\frac{1}{1-\eta-\eta\delta}g\left(n\right)=\sum_{k\ge0}\left(\eta+\eta\delta\right)^{k}g\left(n\right).
\]
Note that the two operators $\eta$ and $\delta$ do not commute.
However, we have
\[
\left(\eta+\eta\delta\right)g\left(n\right)=g\left(2n\right)+g\left(2n+1\right),
\]
\begin{align*}
\left(\eta+\eta\delta\right)^{2}g\left(n\right) & =\left(\eta+\eta\delta\right)\left(g\left(2n\right)+g\left(2n+1\right)\right)\\
 & =g\left(4n\right)+g\left(4n+2\right)+g\left(4n+1\right)+g\left(4n+3\right),
\end{align*}
and more generally, by induction,
\[
\left(\eta+\eta\delta\right)^{k}g\left(n\right)=\sum_{l=0}^{2^{k}-1}g\left(2^{k}n+l\right).
\]
Therefore, we have
\[
f\left(n\right)=\sum_{k\ge0}\sum_{l=0}^{2^{k}-1}g\left(2^{k}n+l\right).
\]
It can be verified by substitution that
\begin{align*}
f\left(n\right) & =g\left(n\right)\\
 & +g\left(2n\right)+g\left(2n+1\right)\\
 & +g\left(4n\right)+g\left(4n+1\right)+g\left(4n+2\right)+g\left(4n+3\right)\\
 & +\cdots
\end{align*}
indeed satisfies (\ref{eq:equation}).
\end{proof}
\begin{rem}
The triple summation is formidable to solve. It is often easiest to deal with $g(n)$ of a form which allows for a high degree of telescoping. For example, in the simple case 
\[
g\left(n\right)=\frac{1}{n^{s}}-\frac{1}{\left(2n\right)^{s}},
\]
the sequence $f\left(n\right)$ is evaluated as the telescoping sum
\begin{align*}
f\left(n\right) & =\sum_{k\ge0}\sum_{l=0}^{2^{k}-1}\frac{1}{\left(2^{k}n+l\right)^{s}}-\frac{1}{\left(2^{k+1}n+l\right)^{s}}=\frac{1}{n^{s}}-\frac{1}{\left(n+1\right)^{s}}.
\end{align*}
Note that Allouche and Shallit \cite{Allouche Shallit-1} previously looked for all eigenfunctions of the operator
\[
f\left(x\right)\mapsto f\left(2x\right)+f\left(2x+1\right).
\]
They show that, under some technical conditions, this same sequence $\frac{1}{n^{s}}-\frac{1}{\left(n+1\right)^{s}}$
is the only solution to
\[
f\left(2x\right)+f\left(2x+1\right)=\lambda f\left(x\right).
\]
\end{rem}

\begin{cor}
More generally, the equation 
\begin{equation}
g\left(x+n\right)=f\left(x+n\right)-f\left(x+2n\right)-f\left(x+2n+1\right)\label{eq:equation-1}
\end{equation}
has formal solution
\[
f\left(x+n\right)=\sum_{k\ge0}\sum_{l=0}^{2^{k}-1}g\left(x+2^{k}n+l\right).
\]
As a consequence,
\begin{align*}
\sum_{n\ge1}s_{2}\left(n\right)g\left(x+n\right) & =\sum_{n\ge0}f\left(x+2n+1\right)\\
 & =\sum_{n\ge0}\sum_{k\ge0}\sum_{l=0}^{2^{k}-1}g\left(x+2^{k+1}n+2^{k}+l\right).
\end{align*}
\end{cor}

\begin{rem}
Choosing $x=0$ and $g\left(n\right)=z^{n}$ in this identity yields
\[
\sum_{n\ge1}s_{2}\left(n\right)z^{n}=\sum_{n\ge0}\sum_{k\ge0}\sum_{l=0}^{2^{k}-1}z^{2^{k+1}n+2^{k}+l};
\]
after summation over $l$ first, then over $n,$ we recover 
\[
\sum_{n\ge1}s_{2}\left(n\right)z^{n}=\frac{1}{1-z}\sum_{k\ge0}\frac{z^{2^{k}}}{1+z^{2^{k}}}
\]
which is the Lambert series (\ref{eq:Lambert infinite}).
\end{rem}

%This is the most general framework we have to turn a summation with $s_2(n)$ into one without $s_2(n)$. If we only allow $g(n)$ to be nonzero for finitely many terms, we then also have a result for finite summations.

\subsection{Recovering $J_{\infty}$}

Our goal is to recover, from result \eqref{eq:f(n)}, the expression
\[
J_{\infty}\left(x\right)=\sum_{n\geq1}\frac{s_{2}(n)}{(x+n)(x+n+1)}=\sum_{l\geq0}\frac{1}{2^{l}}\beta\left(\frac{x}{2^{l}}+1\right).
\]
Using the expansion of the digamma function 
\[
\psi(x+1)=-\gamma+\sum_{k\geq1}\left(\frac{1}{k}-\frac{1}{x+k}\right),
\]
we can express the Stirling $\beta$ function as follows:
\[
\beta(2x+1)=\frac{1}{2}\psi(x+1)-\frac{1}{2}\psi\left(x+\frac{1}{2}\right)=\frac{1}{2}\sum_{k\geq1}\frac{1}{(x+k)(x-\frac{1}{2}+k)}.
\]
Therefore, 
\[
J_{\infty}\left(x\right)=\sum_{l\geq0}\frac{1}{2^{l+2}}\sum_{k\geq1}\frac{1}{(\frac{x}{2^{l+1}}+k)(\frac{x}{2^{l+1}}-\frac{1}{2}+k)}.
\]
Now using the general formula \eqref{eq:f(n)} with
\[
g(n)=\frac{1}{(x+n)(x+n+1)}=\frac{1}{x+n}-\frac{1}{x+n+1}
\]
gives
\[
f(n)=\sum_{k\geq0}\sum_{l=0}^{2^{k}-1}g(2^{k}n+l)=\sum_{k\geq0}\frac{1}{x+2^{k}n}-\frac{1}{x+2^{k}(n+1)},
\]
since the inner sum telescopes. Now we have 
\begin{align*}
J_{\infty}\left(x\right)=\sum_{n\geq1}f(2n-1) & =\sum_{n\geq1}\sum_{k\geq0}\frac{2^{k}}{(x+2^{k}(2n-1))(x+2^{k}(2n))}\\
 & =\sum_{k\geq0}\frac{1}{2^{k+2}}\sum_{n\geq1}\frac{1}{(\frac{x}{2^{k+1}}+n-\frac{1}{2})(\frac{x}{2^{k+1}}+n)},
\end{align*}
which, after relabelling variables, gives the desired result.

\subsection{Extension to arbitrary base $b$}

The extension of the results of the previous section to an arbitrary
base is given next. We denote the sum of digits of $n$ in the base $b$ 
%$a_{j,b}\left(n\right)$ the number of occurrences of the symbol
%$0\le j\le b-1$ in the base$-b$ expansion of $n,$ while 
by $s_{b}\left(n\right)$.
\begin{thm}
The sequence $s_{b}\left(n\right)$ satisfies
\[
\sum_{n\ge1}s_{b}\left(n\right)\left(f\left(n\right)-\sum_{j=0}^{b-1}f\left(bn+j\right)\right)=\sum_{j=1}^{b-1}j\sum_{n\ge0}f\left(bn+j\right).
\]
\end{thm}

\begin{proof}
We thank J.-P. Allouche for notifying us that this result can be deduced directly from the Lemma p. 21 of \cite{Allouche Shallit-1}. We provide an alternate proof which could lead to future generalization. The only useful fact we have is that $s_b(bn+j)=s_b(n)+j$ for $0 \le j \le b-1$, so we attempt to get  something to telescope using this recurrence. Also note that $s_b(0)=0$, so that many of the summations can start at either $n=0$ or $n=1$. Therefore, consider $$ \sum_{n\ge1}s_b(n) f(n) =\sum_{n\ge1} \sum_{j=0}^{b-1} s_b(bn+j)f(bn+j) = \sum_{n\ge1} \sum_{j=0}^{b-1} s_b(n)f(bn+j) + \sum_{n\ge1} \sum_{j=0}^{b-1}j f(bn+j). $$
Subtracting the right terms from both sides gives $$\sum_{n\geq 1}s_b(n) \left( f(n) - \sum_{j=0}^{b-1} f(bn+j)  \right) = \sum_{j=1}^{b-1}j \sum_{n \geq 0} f\left(  bn+j \right).$$
\end{proof}
This suggests a natural extension: consider higher order recurrences for $s_b$, such as $s_b(b^2n+j)=s_b(n)+s_b(j)$ (in this case $j$ only has two digits so we can even give an explicit formula like $s_b(j)=j - (b-1)  \left \lfloor{\frac{j}{b}}\right \rfloor  $). However, nothing along these lines appears to yield a nice closed form expression, and the most natural recursive formula appears to be the one given above.

We can give an explicit solution to this implicit equation. However, in practice this formula is \textit{extremely} difficult to apply, due to the presence of a quadruple summation.
\begin{thm}
\label{thm:Sum base b}The equation
\[
g\left(n\right)=f\left(n\right)-\sum_{j=0}^{b-1}f(bn+j)
\]
has formal solution
\[
f\left(n\right)=\sum_{k\ge0}\sum_{l=0}^{b^{k}-1}g\left(b^{k}n+l\right).
\]
As a consequence,
\[
\sum_{n\ge1}s_{b}\left(n\right)g\left(n\right)=\sum_{j=1}^{b-1}j\sum_{n,k\ge0}\sum_{l=0}^{b^{k}-1}g\left(b^{k+1}n+b^{k}j+l\right).
\]
\end{thm}

\begin{proof}
Introduce the operators 
\[
\eta f\left(n\right)=f\left(bn\right)
\]
and 
\[
\delta f\left(n\right)=f\left(n+\frac{1}{b}\right).
\]
Then rewrite
\[
g\left(n\right)=\left(1-\eta-\eta\delta-\eta\delta^{2}-\cdots-\eta\delta^{b-1}\right)f\left(n\right)
\]
so that 
\[
f\left(n\right)=\frac{1}{1-\eta-\eta\delta-\eta\delta^{2}-\cdots-\eta\delta^{b-1}}g\left(n\right)=\sum_{k\ge0}\left(\eta+\eta\delta+\cdots+\eta\delta^{b-1}\right)^{k}g\left(n\right).
\]
Note again that the two operators do not commute; however, we have 
\[
\left(\eta+\eta\delta+\cdots+\eta\delta^{b-1}\right)g\left(n\right)=\sum_{l=0}^{b-1}g(bn+l),
\]
\[
\left(\eta+\eta\delta+\cdots+\eta\delta^{b-1}\right)^{2}g\left(n\right)=\sum_{j=0}^{b-1}\sum_{l=0}^{b-1}g(b^{2}n+bl+j)=\sum_{l=0}^{b^{2}-1}g(b^{2}n+l),
\]
and by induction 
\[
\left(\eta+\eta\delta+\cdots+\eta\delta^{b-1}\right)^{k}g\left(n\right)=\sum_{l=0}^{b^{k}-1}g\left(b^{k}n+l\right).
\]
Therefore, 
\[
f\left(n\right)=\sum_{k\ge0}\sum_{l=0}^{b^{k}-1}g\left(b^{k}n+l\right).
\]
\end{proof}

\subsection{A finite version}

This corresponds to taking $g(n)=0$ for $n\ge 2^p-1$ in the results of Section \ref{recursivesec}. In this finite case, paying special attention to the upper indices of all of our summations, we write
\begin{align*}
\sum_{n=1}^{2^{p}-1}s_{2}\left(n\right)f\left(n\right) & =\sum_{n=0}^{2^{p-1}-1}s_{2}\left(2n\right)f\left(2n\right)+\sum_{n=0}^{2^{p-1}-1}s_{2}\left(2n+1\right)f\left(2n+1\right)\\
 & =\sum_{n=1}^{2^{p-1}-1}s_{2}\left(n\right)\left(f\left(2n\right)+f\left(2n+1\right)\right)+\sum_{n=0}^{2^{p-1}-1}f\left(2n+1\right).
\end{align*}
We define the sequence $g\left(n\right)$ as follows:
\begin{equation}
g\left(n\right)=\begin{cases}
f\left(n\right)-f\left(2n\right)-f\left(2n+1\right), & 0\le n\le2^{p-1}-1;\\
f\left(n\right), & 2^{p-1}\le n\le2^{p}-1,
\end{cases}\label{eq:g}
\end{equation}
so that we have
\begin{equation}
\sum_{n=1}^{2^{p}-1}s_{2}\left(n\right)g\left(n\right)=\sum_{n=0}^{2^{p-1}-1}f\left(2n+1\right).\label{eq:sumg}
\end{equation}
Given a sequence $\left\{ g\left(n\right)\right\} ,$ solving equation
(\ref{eq:g}) for $f\left(n\right)$ allows for the computation of the sum $\sum_{n=1}^{2^{p}-1}s_{2}\left(n\right)g\left(n\right)$
using (\ref{eq:sumg}).

Looking at the simple case $p=4$ will help us understand how to solve
(\ref{eq:g}). This system of equations is made of the trivial
\[
g\left(n\right)=f\left(n\right),\thinspace\thinspace8\le n\le15
\]
and 
\begin{align*}
g\left(7\right) & =f\left(7\right)-f\left(14\right)-f\left(15\right)\\
g\left(6\right) & =f\left(6\right)-f\left(12\right)-f\left(13\right)\\
\vdots\\
g\left(1\right) & =f\left(1\right)-f\left(2\right)-f\left(3\right).
\end{align*}
The solution is given by
\[
f\left(n\right)=g\left(n\right),\thinspace\thinspace8\le n\le15
\]
and
\begin{align*}
f\left(7\right) & =g\left(7\right)+g\left(14\right)+g\left(15\right)\\
\vdots\\
f\left(4\right) & =g\left(4\right)+g\left(8\right)+g\left(9\right).
\end{align*}
Then
\begin{align*}
f\left(3\right) & =g\left(3\right)+f\left(6\right)+f\left(7\right)\\
 & =g\left(3\right)+g\left(6\right)+g\left(12\right)+g\left(13\right)+g\left(7\right)+g\left(14\right)+g\left(15\right),
\end{align*}
\begin{align*}
f\left(2\right) & =g\left(2\right)+f\left(4\right)+f\left(5\right)\\
 & =g\left(2\right)+g\left(4\right)+g\left(8\right)+g\left(9\right)+g\left(5\right)+g\left(10\right)+g\left(11\right),
\end{align*}
and
\begin{align*}
f\left(1\right) & =g\left(1\right)+f\left(2\right)+f\left(3\right)\\
 & =g\left(1\right)\\
 & +g\left(2\right)+g\left(4\right)+g\left(8\right)+g\left(9\right)+g\left(5\right)+g\left(10\right)+g\left(11\right)\\
 & +g\left(3\right)+g\left(6\right)+g\left(12\right)+g\left(13\right)+g\left(7\right)+g\left(14\right)+g\left(15\right).
\end{align*}

The following result can be proved by induction on $p.$
\begin{thm}
\label{thm:Finite sum}For $p$ a fixed integer, the solution to the
system of equations (\ref{eq:g}) is given by
\[
f\left(n\right)=\sum_{k=0}^{p-l_{n}}\sum_{l=0}^{2^{k}-1}g\left(2^{k}n+l\right)
\]
where $l_n =  \left \lfloor{ \log_2 n  }\right \rfloor$ is the number of bits in the binary representation
of $n.$

As a consequence, 
\begin{align}
\sum_{n=1}^{2^{p}-1}s_{2}\left(n\right)g\left(n\right)=\sum_{n=1}^{2^{p-1}-1}\sum_{k=0}^{p-l_{2n+1}}\sum_{l=0}^{2^{k}-1}g\left(2^{k+1}n+2^{k}+l\right).
\end{align}
\end{thm}

\subsection{A last result}
To state our last result, we introduce the Barnes zeta functions - see for example \cite{Ruijsenaars} - defined by
\[
\zeta_{p}\left(\alpha;x;\left(a_{1},\dots,a_{p}\right)\right)=\sum_{m_{1},\dots,m_{p} \ge 0}\left(x+a_{1}m_{1}+\dots+a_{p}m_{p}\right)^{-\alpha}
\]
%with the vector $\mathbf{a}=\left(a_{1},\dots,a_{p}\right).$
where the variables $x$ and $a_i$ are positive, and $\alpha > p.$
\begin{thm}
\label{thm29}
The following Hurwitz-type generating function can be computed explicitly: for $\alpha>2$ and $z \ge 0,$
\begin{align}
\label{finite zeta}
\sum_{n=1}^{b^{p}-1}s_{b}\left(n\right)\frac{1}{\left(n+z\right)^{\alpha}} & =\sum_{l=0}^{p-1}\left[\zeta_{2}\left(\alpha,z+b^{l},\left(1,b^{l}\right)\right)-\zeta_{2}\left(\alpha,z+b^{l}+b^{p},\left(1,b^{l}\right)\right)\right]\\
 & -b\sum_{l=1}^{p}\left[\zeta_{2}\left(\alpha,z+b^{l},\left(1,b^{l}\right)\right)-\zeta_{2}\left(\alpha,z+b^{l}+b^{p},\left(1,b^{l}\right)\right)\right].\nonumber
\end{align}
The case $p \to \infty$ is 
\begin{align}
\label{infinite zeta}
\sum_{n=1}^{+\infty}s_{b}\left(n\right)\frac{1}{\left(n+z\right)^{\alpha}} 
%& =
%\sum_{l=0}^{+\infty}\zeta_{2}\left(\alpha,z+b^{l},\left(1,b^{l}\right)\right)-b\sum_{l=1}^{+\infty}\zeta_{2}\left(\alpha,z+b^{l},\left(1,b^{l}\right)\right)\\
 & =-z\zeta\left(\alpha,z+1\right)+\zeta\left(\alpha-1,z+1\right)+\left(1-b\right)\sum_{l=1}^{+\infty}\zeta_{2}\left(\alpha,z+b^{l},\left(1,b^{l}\right)\right).
\end{align}
\end{thm}

\begin{proof}
This identity can be obtained from Thm \ref{thm:Finite sum}; however, this would be a non-rigorous derivation. Therefore, we provide a rigorous proof by induction on the value of $p.$ We  consider here the case $b=2$ since the proof is identical for other values of $b.$  

The base case $p=1$ is as follows: the left-hand side in \eqref{finite zeta} is simply
$
\frac{1}{\left(z+1\right)^{2}},
$
while the right-hand side is
\begin{align*}
\zeta_{2}\left(\alpha,z+1,\left(1,1\right)\right)-\zeta_{2}\left(\alpha,z+1+2,\left(1,1\right)\right)
-2\left(\zeta_{2}\left(\alpha,z+2,\left(1,2\right)\right)-\zeta_{2}\left(\alpha,z+4,\left(1,2\right)\right)\right).
\end{align*}
The first difference is
\[
\zeta_{2}\left(\alpha,z+1,\left(1,1\right)\right)-\zeta_{2}\left(\alpha,z+1+2,\left(1,1\right)\right)
=\sum_{n,p\ge0}\frac{1}{\left(z+1+n+p\right)^{\alpha}}-\frac{1}{\left(z+1+\left(n+1\right)+\left(p+1\right)\right)^{\alpha}}.
\]
Through $2$-dimensional telescoping, each $\left(n,p\right)$
term cancels the next $\left(n+1,p+1\right)$ term so that only the
boundary terms
\[
\sum_{n,p\ge0}\frac{1}{\left(z+1+n+p\right)^{\alpha}}-\frac{1}{\left(z+1+\left(n+1\right)+\left(p+1\right)\right)^{\alpha}}=2\zeta\left(\alpha,z+1\right)-\frac{1}{\left(z+1\right)^{\alpha}}
\]
remain. Similarly, the second difference
\begin{align*}
&\zeta_{2}\left(\alpha,z+2,\left(1,2\right)\right)-\zeta_{2}\left(\alpha,z+4,\left(1,2\right)\right)=\sum_{n,p\ge0}\frac{1}{\left(z+2+n+2p\right)^{\alpha}}-\frac{1}{\left(z+2+\left(n+1\right)+2\left(p+1\right)\right)^{\alpha}}
\end{align*}
telescopes to give $
-2\zeta\left(\alpha,z+2\right).$
The sum of the two terms gives
\[
2\zeta\left(\alpha,z+1\right)-\frac{1}{\left(z+1\right)^{\alpha}}-2\zeta\left(\alpha,z+2\right)=\frac{2}{\left(z+1\right)^{\alpha}}-\frac{1}{\left(z+1\right)^{\alpha}}=\frac{1}{\left(z+1\right)^{\alpha}},
\]
which proves the case $p=1.$ 

As $p$ is replaced by $p+1,$ the left hand side in \eqref{finite zeta} becomes
\[
\sum_{n=0}^{2^{p+1}-1}\frac{s_{2}\left(n\right)}{\left(z+n\right)^{\alpha}}  =\sum_{n=0}^{2^{p}-1}\frac{s_{2}\left(n\right)}{\left(z+n\right)^{\alpha}}+\sum_{n=2^{p}}^{2^{p+1}-1}\frac{s_{2}\left(n\right)}{\left(z+n\right)^{\alpha}}
=\sum_{n=0}^{2^{p}-1}\frac{s_{2}\left(n\right)}{\left(z+n\right)^{\alpha}}+\frac{s_{2}\left(n+2^{p}\right)}{\left(z+n+2^{p}\right)^{\alpha}}.
\]
Since, for $0\le n\le2^{p}-1,$
\[
s_{2}\left(n+2^{p}\right)=s_{2}\left(n\right)+1,
\]
we deduce, using the induction hypothesis, that
\begin{align*}
\sum_{n=0}^{2^{p+1}-1}\frac{s_{2}\left(n\right)}{\left(z+n\right)^{\alpha}} & =\sum_{n=0}^{2^{p}-1}\left(\frac{s_{2}\left(n\right)}{\left(z+n\right)^{\alpha}}+\frac{s_{2}\left(n\right)+1}{\left(z+n+2^{p}\right)^{\alpha}}\right)\\
 & =f_{p}\left(z\right)+f_{p}\left(z+2^{p}\right)+\sum_{n=0}^{2^{p}-1}\frac{1}{\left(z+n+2^{p}\right)^{\alpha}}.
\end{align*}
We use the shorthand notation 
\[
f_{p}\left(z\right)=\sum_{l=0}^{p-1}h_{l}\left(z+2^{l}\right)-h_{l}\left(z+2^{l}+2^{p}\right)-2\sum_{l=1}^{p}h_{l}\left(z+2^{l}\right)-h_{l}\left(z+2^{l}+2^{p}\right)
\]
to denote the right-hand side of \eqref{finite zeta}, with 
\[
h_{l}\left(z\right)=\zeta_{2}\left(\alpha,z,\left(1,2^{l}\right)\right).
\]
However,
\begin{align*}
f_{p}\left(z\right)+f_{p}\left(z+2^{p}\right) & =\sum_{l=0}^{p-1}h_{l}\left(z+2^{l}\right)-h_{l}\left(z+2^{l}+2^{p}\right)-2\sum_{l=1}^{p}h_{l}\left(z+2^{l}\right)-h_{l}\left(z+2^{l}+2^{p}\right)\\
 & +\sum_{l=0}^{p-1}h_{l}\left(z+2^{l}+2^{p}\right)-h_{l}\left(z+2^{l}+2^{p+1}\right)-2\sum_{l=1}^{p}h_{l}\left(z+2^{l}+2^{p}\right)-h_{l}\left(z+2^{l}+2^{p+1}\right)\\
 & =\sum_{l=0}^{p-1}h_{l}\left(z+2^{l}\right)-h_{l}\left(z+2^{l}+2^{p+1}\right)-2\sum_{l=1}^{p}h_{l}\left(z+2^{l}\right)-h_{l}\left(z+2^{l}+2^{p+1}\right)
\end{align*}
These two sums are  extended to sums over $\left[0,p \right]$ and $\left[0,p+1 \right]$ respectively by adding one extra term to obtain
\begin{align*} 
 & \sum_{l=0}^{p}h_{l}\left(z+2^{l}\right)-h_{l}\left(z+2^{l}+2^{p+1}\right)-2\sum_{l=1}^{p+1}h_{l}\left(z+2^{l}\right)-h_{l}\left(z+2^{l}+2^{p+1}\right)\\
 & -\left(h_{p}\left(z+2^{p}\right)-h_{p}\left(z+2^{p}+2^{p+1}\right)\right)+2\left(h_{p+1}\left(z+2^{p+1}\right)-h_{p+1}\left(z+2^{p+2}\right)\right)\\
 & =f_{p+1}\left(z\right)
 -\left(h_{p}\left(z+2^{p}\right)-h_{p}\left(z+2^{p}+2^{p+1}\right)\right)+2\left(h_{p+1}\left(z+2^{p+1}\right)-h_{p+1}\left(z+2^{p+2}\right)\right).
\end{align*}
We thus only need to prove that
\[
f_{p+1}\left(z\right)=f_{p}\left(z\right)+f_{p}\left(z+2^{p}\right)+\sum_{n=0}^{2^{p}-1}\frac{1}{\left(z+n+2^{p}\right)^{\alpha}}
\]
or, equivalently, that
\begin{align*}
\sum_{n=0}^{2^{p}-1}\frac{1}{\left(z+n+2^{p}\right)^{\alpha}} 
%& =
%h_{p}\left(z+2^{p}\right)-h_{p}\left(z+2^{p}+2^{p+1}\right)-2h_{p+1}\left(z+2^{p+1}\right)+2h_{p+1}\left(z+2^{p+2}\right)\\
 & =\sum_{r,s\ge0}\left(z+2^{p}+r+2^{p}s\right)^{-\alpha}-\left(z+2^{p}+2^{p+1}+r+2^{p}s\right)^{-\alpha}\\
 & -2\left(z+2^{p+1}+r+2^{p+1}s\right)^{-\alpha}+2\left(z+2^{p+2}+r+2^{p+1}s\right)^{-\alpha}.
\end{align*}
The first two terms are
\begin{align*}
&\sum_{r,s\ge0}\left(z+2^{p}+r+2^{p}s\right)^{-\alpha}-\left(z+2^{p}+2^{p+1}+r+2^{p}s\right)^{-\alpha}\\
& =
\sum_{r\ge0,s\ge0}\left(z+2^{p}+r+2^{p}s\right)^{-\alpha}-\sum_{r\ge0,s\ge0}\left(z+2^{p}+r+2^{p}\left(s+2\right)\right)^{-\alpha}\\
& =
\sum_{r\ge0,s\ge0}\left(z+2^{p}+r+2^{p}s\right)^{-\alpha}-\sum_{r\ge0,s\ge2}\left(z+2^{p}+r+2^{p}s\right)^{-\alpha} \\
& =\sum_{r\ge0}\left(z+2^{p}+r\right)^{-\alpha}+\sum_{r\ge0}\left(z+2^{p+1}+r\right)^{-\alpha},
\end{align*}
while the last two terms are
\begin{align*}
&-2\sum_{r,s\ge0}\left(z+2^{p+1}+r+2^{p+1}s\right)^{-\alpha}-\left(z+2^{p+1}+2^{p+1}+r+2^{p+1}s\right)^{-\alpha}\\ 
& =
-2\sum_{r,s\ge0}\left(z+2^{p+1}+r+2^{p+1}s\right)^{-\alpha}-\left(z+2^{p+1}+r+2^{p+1}\left(s+1\right)\right)^{-\alpha}
\\ & =
-2\sum_{r,s\ge0}\left(z+2^{p+1}+r+2^{p+1}s\right)^{-\alpha}+2\sum_{r\ge0,s\ge1}\left(z+2^{p+1}+r+2^{p+1}s\right)^{-\alpha}\\
& = -2\sum_{r\ge0}\left(z+2^{p+1}+r\right)^{-\alpha}.
\end{align*}
We deduce
\begin{align*}
& h_{p}\left(z+2^{p}\right)-h_{p}\left(z+2^{p}+2^{p+1}\right)-2h_{p+1}\left(z+2^{p+1}\right)+2h_{p+1}\left(z+2^{p+2}\right)\\ 
& =
\sum_{r\ge0}\left(z+2^{p}+r\right)^{-\alpha}+\sum_{r\ge0}\left(z+2^{p+1}+r\right)^{-\alpha}-2\sum_{r\ge0}\left(z+2^{p+1}+r\right)^{-\alpha} \\
& =
\sum_{r\ge0}\left(z+2^{p}+r\right)^{-\alpha}-\sum_{r\ge0}\left(z+2^{p+1}+r\right)^{-\alpha} \\
& =
\sum_{r\ge0}\left(z+2^{p}+r\right)^{-\alpha}-\sum_{r\ge0}\left(z+2^{p}+r+2^{p}\right)^{-\alpha} \\
& =
\sum_{r\ge0}\left(z+2^{p}+r\right)^{-\alpha}-\sum_{r\ge2^{p}}\left(z+2^{p}+r\right)^{-\alpha} \\
& =
\sum_{r\ge0}^{2^{p}-1}\left(z+2^{p}+r\right)^{-\alpha},
\end{align*}
which is the desired result.
The asymptotic case $p \to \infty$ is deduced from \eqref{finite zeta} using the same approach as in Thm \ref{thm:Thm3} and the easily checked specific value of the Barnes zeta function
\begin{align*}
\zeta_{2}\left(\alpha,z+1,\left(1,1\right)\right) 
%& =\frac{\left(-1\right)^{\alpha-1}}{\left(\alpha-1\right)!}\left[\left(\alpha-1\right)\psi^{\left(\alpha-2\right)}\left(z\right)-\psi^{\left(\alpha-1\right)}\left(z\right)+z\psi^{\left(\alpha-1\right)}\left(z\right)\right]\\
=-z\zeta\left(\alpha,z+1\right)+\zeta\left(\alpha-1,z+1\right).
\end{align*}
\end{proof}

The limit case $\alpha=2$ in Thm \ref{thm29} can be obtained in terms of Barnes' double
Gamma function $\Gamma_{2}\left(z,\left(\omega_{1},\omega_{2}\right)\right)$, which is defined
as the derivative
\[
\Gamma_{2}\left(z,\left(\omega_{1},\omega_{2}\right)\right)=\frac{\partial}{\partial\alpha}\zeta_{2}\left(\alpha,z,\left(\omega_{1},\omega_{2}\right)\right)_{\vert\alpha=0}.
\]
Barnes shows \cite[p. 286]{barnes} that the third order logarithmic derivative of the double log-Gamma function satisties 
\[
\psi_{2}^{\left(3\right)}\left(z,\left(\omega_{1}\omega_{2}\right)\right)=\frac{d^{3}}{dz^{3}}\log\Gamma_{2}\left(z,\left(\omega_{1},\omega_{2}\right)\right)=-2\sum_{m_{1},m_{2}\ge0}\frac{1}{\left(z+m_{1}\omega_{1}+m_{2}\omega_{2}\right)^{3}},
\]
so that 
\[
\lim_{\alpha\to2}\sum_{m_{1},m_{2}\ge0}\frac{1}{\left(z+m_{1}\omega_{1}+m_{2}\omega_{2}\right)^{\alpha}}=\psi_{2}^{\left(2\right)}\left(z,\left(\omega_{1},\omega_{2}\right)\right).
\]
As a consequence, 
%\begin{align*}
%\sum_{n \ge 1}\frac{s_{b}\left(n\right)}{\left(n+z\right)^{2}} 
%%& =\sum_{l=0}^{p}\psi_{2}^{\left(2\right)}\left(z+b^{l},\left(1,b^{l}\right)\right)\\
%% & -b\sum_{l=1}^{+\infty}\left[\psi_{2}^{\left(2\right)}\left(z+b^{l},\left(1,b^{l}\right)\right)-\psi_{2}^{\left(2\right)}\left(z+b^{l}+b^{p},\left(1,b^{l}\right)\right)\right]\\
% & =\psi_{2}^{\left(2\right)}\left(z,\left(1,1\right)\right)+\left(1-b\right)\sum_{l=1}^{+\infty}\psi_{2}^{\left(2\right)}\left(z,\left(1,b^{l}\right)\right).
%\end{align*}
%and the asymptotic case is
\begin{align*}
\sum_{n\ge1}\frac{s_{b}\left(n\right)}{\left(n+z\right)^{2}} 
%& =\sum_{l=0}^{+\infty}\psi_{2}^{\left(2\right)}\left(z,\left(1,b^{l}\right)\right)-b\sum_{l=1}^{+\infty}\psi_{2}^{\left(2\right)}\left(z,\left(1,b^{l}\right)\right).\\
 & =\psi_{2}^{\left(2\right)}\left(z,\left(1,1\right)\right)+\left(1-b\right)\sum_{l=1}^{+\infty}\psi_{2}^{\left(2\right)}\left(z,\left(1,b^{l}\right)\right).
\end{align*}
An explicit computation of the $\psi_{2}^{\left(2\right)}$ function yields the following result.
\begin{cor}\label{cor30}
The following identity holds for $z>0$:
\begin{align*}
\sum_{n\ge1}\frac{s_{b}\left(n\right)}{\left(n+z\right)^{2}}  =
%\psi_{2}^{\left(2\right)}\left(z,\left(1,1\right)\right)+\left(1-b\right)\sum_{l=1}^{+\infty}\psi_{2}^{\left(2\right)}\left(z,%\left(1,b^{l}\right)\right)\\&=
 -\psi\left(z\right)+ \left(1-z\right)\psi'\left(z\right)
 %\sum_{m\ge0}\left[\zeta\left(2,m+z\right)-\frac{1}{m+z}\right]\\
+\left(1-b\right)\sum_{l\ge1}\left( -1-\psi\left(z\right)+\frac{1}{b^{2l}}\sum_{m\ge0}\tilde{\psi'}\left(mb^{l}+z\right)\right) 
\end{align*}
with
\[
\tilde{\psi'}\left(z\right) = \psi'\left(z\right)-\frac{1}{z}.
\]
\end{cor}
\begin{proof}

The explicit expression of the $\psi_{2}^{\left(2\right)}$ function
is provided by Spreafico \cite{Spreafico}, where the residue at the simple pole at $\alpha=2$ and corresponding finite
part of the double zeta function are given by
\begin{align*}
\zeta_{2}\left(\alpha,z,\left(\omega_{1},\omega_{2}\right)\right) & \underset{\alpha=2}{\sim}\frac{1}{\omega_{1}\omega_{2}}\frac{1}{\alpha-2}\\
& -\frac{1}{\omega_{1}\omega_{2}}\left(1+\log\omega_{1}+\psi\left(\frac{z}{\omega_{1}}\right)\right) +\frac{1}{\omega_{2}^{2}}\sum_{m\ge0}\left[\zeta\left(2,\frac{\omega_{2}m+z}{\omega_{1}}\right)-\frac{\omega_{1}}{\omega_{2}m+z}\right].
\end{align*}
Substituting in \eqref{infinite zeta} shows that the residue at $\alpha=2$ cancels and
only the finite part remains, which yields
\begin{align*}
\sum_{n\ge1}\frac{s_{b}\left(n\right)}{\left(n+z\right)^{2}}  &=
 \sum_{m\ge0}\left[\zeta\left(2,m+z\right)-\frac{1}{m+z}\right]\\
&+\left(1-b\right)\sum_{l\ge1}\left\{ -1-\psi\left(z\right)+\frac{1}{b^{2l}}\sum_{m\ge0}\left[\zeta\left(2,mb^{l}+z\right)-\frac{1}{mb^{l}+z}\right]\right\} .
\end{align*}
The first sum on the right-hand side can be simplified to
\[
\sum_{m\ge0}\left[\zeta\left(2,m+z\right)-\frac{1}{m+z}\right]=\left(1-z\right)\psi'\left(z\right)+1
\]
as follows: since $\zeta\left(2,z\right)=\psi'\left(z\right),$ consider the finite sum
\begin{align*}
\sum_{m=0}^{M}\psi'\left(z+m\right)-\frac{1}{m+z} & =\sum_{m=0}^{M}\left[\psi'\left(z+m\right)-\psi\left(m+z+1\right)+\psi\left(m+z\right)\right]\\
 & =\sum_{m=0}^{M}\psi'\left(z+m\right)-\left(\psi\left(M+z+1\right)-\psi\left(z\right)\right).
\end{align*}
Since moreover, by basic properties of the digamma function,
\[
\sum_{m=0}^{M}\psi\left(z+m\right)=-M-1+\left(1-z\right)\psi\left(z\right)+\left(M+z\right)\psi\left(M+z+1\right),
\]
so that, by differentiation,
\[
\sum_{m=0}^{M}\psi'\left(z+m\right)=-\psi\left(z\right)+\psi\left(M+z+1\right)+\left(1-z\right)\psi'\left(z\right)+\left(M+z\right)\psi'\left(M+z+1\right).
\]
and
\[
\sum_{m=0}^{M}\psi'\left(z+m\right)-\frac{1}{m+z}=\left(1-z\right)\psi'\left(z\right)+\left(M+z\right)\psi'\left(M+z+1\right).
\]
The limit
\[
\lim_{M\to+\infty}\left(M+z\right)\psi'\left(M+z+1\right)=\lim_{z\to+\infty}z\psi'\left(z+1\right)=1,
\]
as a consequence of the asymptotic behavior 
\[
\psi'\left(z\right)=\frac{1}{z}+\frac{2}{z^2}+O\left(\frac{1}{z^3}\right)
\]
as $z \to + \infty$, yields the result.

\end{proof}

As a last remark, computing $\sum_{n=1}^{b^{p}-1}\frac{s_{b}\left(n\right)}{\left(n+z\right)^{\alpha}}$ from the difference 
$\sum_{n=1}^{b^{p}-1}s_{b}\left(n\right)\left(\frac{1}{\left(z+n\right)^{\alpha}}-\frac{1}{\left(z+n+1\right)^{\alpha}}\right)$ amounts to solving a difference equation; we could have used a general result by Ruijsenaars \cite{Ruijsenaars} about minimal solutions of difference equations to obtain this result. 

\section{Conclusion and open questions}
This article provides several instances of finite sums involving the digit
function $s_{2}\left(n\right).$ Its aim is to provide some basic
tools that allow us to compute these sums; many other instances can be deduced from these basic methods.

There are two obvious ways to generalize our results: one is to consider the sum $s_b(n)$ of the base-$b$ digits of $n$, and the associated sequence $\{(-1)^{s_b(n)}\}$. This was accomplished in many of our theorems, which shows that some results transfer from base $2$ to base $b$ almost identically. The second generalization is to \textit{$k$-automatic sequences}, which are characterized by a finite set of recurrences. For example, the Rudin-Shapiro sequence $\{b_n\}$ is completely characterized by the recurrences $b_{2n}=b_n$, $b_{2n+1}= (-1)^nb_n$, and initial values. Many of the proofs given in this paper depend on the recursive properties of $s_2(n)$, and may generalize to related automatic sequences.

Another remark is that we can write identity (\ref{eq:Sn Thm}) in Thm \ref{thm:Thm1} under the form
\begin{equation}
\sum_{n=0}^{2^{N+1}-1}\left(-1\right)^{s_{2}\left(n\right)}2^{-\frac{N\left(N+1\right)}{2}}\left(-\Delta\right)^{-N-1}f\left(x+n\right)=\mathbb{E}f\left(x+Z_{N}\right),\label{eq:Sn Thm 2-2}
\end{equation}
where $\mathbb{E}$ is the expectation operator, and consider the standardized (zero mean and unit variance) version of the random variable $Z_{N}$
\[
\tilde{Z}_{N}=\frac{Z_{N}-\mu_{N}}{\sigma_{N}}.
\]
The mean and variance of $Z_N$ are respectively
\[
\mu_{N}=\mathbb{E}Z_{N}=2^{N}-\frac{N}{2}-1,
\]
\[
\sigma_{N}^{2}=\frac{1}{9}\left(2^{2N}-\frac{3}{4}N-1\right).
\]
This produces
\[
\sum_{n=0}^{2^{N+1}-1}\left(-1\right)^{s_{2}\left(n\right)}2^{-\frac{N\left(N+1\right)}{2}}\left(-\Delta\right)^{-N-1}f\left(x+\frac{n-\mu_{N}}{\sigma_{N}}\right)=\mathbb{E}f\left(x+\tilde{Z}_{N}\right).
\]
This identity suggests that, as $N\to\infty,$ one could expect a formula
of the type
\[
\lim_{N\to\infty}\sum_{n=0}^{2^{N+1}-1}\left(-1\right)^{s_{2}\left(n\right)}2^{-\frac{N\left(N+1\right)}{2}}\left(-\Delta\right)^{-N-1}f\left(x+\frac{n-\mu_{N}}{\sigma_{N}}\right)=\int_{-\infty}^{+\infty}f\left(x+z\right)d\lambda\left(z\right)
\]
for some limit measure $\lambda$ that is interesting to look at.
A detailed analysis of the random variables $Z_{N}$ and $\tilde{Z}_{N}$
is given in the Annex. Surprisingly, all cumulants of the standardized
random variable $\tilde{Z}_{N}$ can be computed explicitly and are
given by (\ref{eq:cumulants}). Their expression shows that the limit
distribution $\lambda$ of $\tilde{Z}_{\infty}$ is not Gaussian.
We leave as an open question whether this distribution has a density or not.

Finally, the possibility to write an identity such as \eqref{eq:Sn Thm 2-2} is related to the theory of discrete splines: the reader is referred to \cite{Dahmen}.

\section{Annex: A study of the associated random variable}

The random variable $Z_{N}$ with probability weights $p_{k}^{\left(N\right)}$
can be characterized as follows: denote a random discrete
variable uniformly distributed over $\left\{ 0,1,\ldots,2^{i}-1\right\} $ by 
$V_{i}$
so that
\[
\Pr\left\{ V_{i}=k\right\} =\begin{cases}
2^{-i}, & 0\le k\le2^{i}-1;\\
0, & \text{otherwise,}
\end{cases}
\]
and define the random variable
\[
Z_{N}=V_{1}+\cdots+V_{N}.
\]
Then
\[
\Pr\left\{ Z_{N}=k\right\} =\begin{cases}
p_{k}^{\left(N\right)}, & 0\le k\le2^{N}-N-1;\\
0, & \text{otherwise.}
\end{cases}.
\]
For example,
\[
\Pr\left\{ V_{1}=0\right\} =\Pr\left\{ V_{1}=1\right\} =\frac{1}{2}
\]
and
\[
\Pr\left\{ V_{2}=0\right\} =\Pr\left\{ V_{2}=1\right\} =\Pr\left\{ V_{2}=2\right\} =\Pr\left\{ V_{2}=3\right\} =\frac{1}{4},
\]
so that
\[
Z_{2}=V_{1}+V_{2}
\]
takes values in the set $\left\{ 0,1,2,3,4\right\} $ with respective
probabilities $\left\{ 1,2,2,2,1\right\} .$
\begin{prop}
The moment generating function of $Z_{N}$ is
\begin{align}
\mathbb{E}e^{zZ_{N}} & =\prod_{i=0}^{N-1}\left(\frac{1}{2}+\frac{1}{2}e^{2^{i}z}\right)^{N-i}.\label{eq:mgf1}
\end{align}
It is also equal to
\begin{equation}
\mathbb{E}e^{zZ_{N}}=\prod_{k=1}^{N}\frac{1}{2^{k}}\frac{1-e^{2^{k}z}}{1-e^{z}}.\label{eq:mgf2}
\end{equation}
\end{prop}

\begin{proof}
Each random variable $V_{i}$ has moment generating function
\[
\mathbb{E}e^{zV_{k}}=\frac{1}{2^{k}}\sum_{k=1}^{2^{k}-1}e^{zk}=\frac{1}{2^{k}}\frac{1-e^{z2^{k}}}{1-e^{z}},
\]
so that
\[
\mathbb{E}e^{zZ_{N}}=\prod_{k=1}^{N}\frac{1}{2^{k}}\frac{1-e^{2^{k}z}}{1-e^{z}}.
\]
From the identity
\[
\frac{1-e^{2^{k}z}}{1-e^{z}}=\prod_{i=0}^{k-1}\left(1+e^{2^{i}z}\right),
\]
we deduce 
\begin{align*}
\prod_{k=1}^{N}\frac{1-e^{2^{k}z}}{1-e^{z}} & =\prod_{k=1}^{N}\prod_{i=0}^{k-1}\left(1+e^{2^{i}z}\right)=\prod_{i=0}^{N-1}\prod_{k=i+1}^{N}\left(1+e^{2^{i}z}\right)\\
 & =\prod_{i=0}^{N-1}\left(1+e^{2^{i}z}\right)^{N-i}.
\end{align*}
\end{proof}
From the moment generating function, we deduce the following:
\begin{prop}
A stochastic representation of $Z_N$ is
\begin{equation}
Z_{N}=\sum_{k=1}^{N}W_{k}\label{eq:sr1},
\end{equation}
where $W_{k}$ are independent, and each $W_{k}$ has discrete uniform
distribution
\[
\Pr\left\{ W_{k}=l\right\} =\begin{cases}
\frac{1}{2^{k}}, & 0\le l\le2^{k}-1;\\
0, & \text{otherwise.}
\end{cases}
\]
Equivalently,
\begin{align*}
Z_{N} & =V_{0}^{\left(0\right)}+\left(V_{0}^{\left(1\right)}+2V_{1}^{\left(1\right)}\right)+\left(V_{0}^{\left(2\right)}+2V_{1}^{\left(2\right)}+4V_{2}^{\left(2\right)}\right)+\cdots\\
 & +\left(V_{0}^{\left(N-1\right)}+2V_{1}^{\left(N-1\right)}+\cdots+2^{N-1}V_{N-1}^{\left(N-1\right)}\right)
\end{align*}
where the $V_{i}^{\left(k\right)}$ are independent Bernoulli random variables
with
\[
\Pr\left\{ V_{i}^{\left(k\right)}=0\right\} =\Pr\left\{ V_{i}^{\left(k\right)}=1\right\} =\frac{1}{2}.
\]
\end{prop}

\begin{proof}
This follows from the characteristic functions (\ref{eq:mgf1}) and
(\ref{eq:mgf2}).
\end{proof}
These stochastic representations allow an easy computation of the
first moments of $Z_{N}$.
\begin{prop}
The random variable $Z_{N}$ has mean
\[
\mu_{N}=\mathbb{E}Z_{N}=2^{N}-\frac{N}{2}-1
\]
and variance
\[
\sigma_{N}^{2}=\frac{1}{9}\left(2^{2N}-\frac{3}{4}N-1\right).
\]
\end{prop}

\begin{proof}
Using the stochastic representation (\ref{eq:sr1}) and the fact that
\[
\mathbb{E}W_{k}=\frac{2^{k}-1}{2},
\]
we deduce
\[
\mathbb{E}Z_{N}=\sum_{k=1}^{N}\frac{2^{k}-1}{2}=2^{N}-\frac{N}{2}-1.
\]
The variance is computed as follows:
\[
\sigma_{N}^{2}=\sum_{k=1}^{N}\sigma_{W_{k}}^{2}
\]
with 
\[
\sigma_{W_{k}}^{2}=\frac{2^{2k}-1}{12}
\]
so that
\[
\sigma_{W_{k}}^{2}=\sum_{k=1}^{N}\frac{2^{2k}-1}{12}=\frac{1}{9}\left(2^{2N}-\frac{3}{4}N-1\right).
\]
\end{proof}
As a consequence, the standardized random variable
\begin{equation}
\hat{Z}_{N}=\frac{Z_{N}-\mu_{N}}{\sigma_{N}}\label{eq:standard}
\end{equation}
has zero mean and unit variance. It can be further characterized as 
follows.
\begin{prop}
The moment generating function of the standardized random variable
$\hat{Z}_{N}$ is
\begin{equation}
\mathbb{E}e^{z\hat{Z}_{N}}=\mathbb{E}e^{\frac{z}{\sigma_{N}}\left(Z_{N}-\mu_{N}\right)}=\prod_{k=1}^{N}\left(\frac{1}{2^{k}}\frac{\sinh\left(\frac{z}{2\sigma_{N}}2^{k}\right)}{\sinh\left(\frac{z}{2\sigma_{N}}\right)}\right)=\prod_{1\le l\le k}^{N}\cosh\left(2^{k-l}\frac{z}{2\sigma_{N}}\right).\label{eq:cf}
\end{equation}
Its cumulant generating function is
\[
\log\varphi_{\hat{Z}_{N}}\left(z\right)=\sum_{k=1}^{N}\log\left(\frac{1}{2^{k}}\frac{\sinh\left(\frac{z}{2\sigma_{N}}2^{k}\right)}{\sinh\left(\frac{z}{2\sigma_{N}}\right)}\right)=\sum_{n\ge1}\kappa_{2n}^{\left(N\right)}\frac{z^{2n}}{2n!},
\]
so that its cumulants are given by
\begin{equation}
\kappa_{2n}^{\left(N\right)}=\frac{9^{n}}{\left(4^{N}-\frac{3}{4}N-1\right)^{n}}\frac{B_{2n}}{2n}\frac{4^{n}\left(4^{nN}-N-1\right)+N}{4^{n}-1}.\label{eq:cumulants}
\end{equation}
\end{prop}

\begin{proof}
Using 
\[
\frac{1}{2^{k}}\frac{\sin\left(2^{k}x\right)}{\sin x}=\prod_{l=1}^{k}\cos\left(2^{k-l}x\right),
\]
we deduce
\[
\mathbb{E}e^{z\hat{Z}_{N}}=\prod_{1\le l\le k}^{N}\cosh\left(2^{k-l}\frac{z}{2\sigma_{N}}\right).
\]
The cumulant generating function is
\[
\log\varphi_{\hat{Z}_{N}}\left(z\right)=\sum_{k=1}^{N}\log\left(\frac{1}{2^{k}}\frac{\sinh\left(\frac{z}{2\sigma_{N}}2^{k}\right)}{\sinh\left(\frac{z}{2\sigma_{N}}\right)}\right)=\sum_{n\ge1}\kappa_{2n}^{\left(N\right)}\frac{z^{2n}}{2n!}.
\]
Using
\[
\log\sinh\left(az\right)=\log a+\sum_{n\ge1}\frac{B_{2n}}{2n}a^{2n}\frac{z^{2n}}{2n!}
\]
we deduce
\begin{align*}
\log\varphi_{\hat{Z}_{N}}\left(z\right) & =\sum_{k=1}^{N}\sum_{n\ge1}\frac{B_{2n}}{2n}\frac{\left(2^{2nk}-1\right)}{\sigma_{N}^{2n}}\frac{z^{2n}}{2n!}\\
 & =\sum_{n\ge1}\frac{B_{2n}}{2n}\left(\sum_{k=1}^{N}\frac{\left(2^{2nk}-1\right)}{\sigma_{N}^{2n}}\right)\frac{z^{2n}}{2n!}\\
 & =\sum_{n\ge1}\frac{B_{2n}}{2n}\frac{4^{n}\left(4^{nN}-N-1\right)+N}{\sigma_{N}^{2n}\left(4^{n}-1\right)}\frac{z^{2n}}{2n!}
\end{align*}
so that the cumulants are
\[
\kappa_{2n}^{\left(N\right)}=\frac{9^{n}}{\left(4^{N}-\frac{3}{4}N-1\right)^{n}}\frac{B_{2n}}{2n}\frac{4^{n}\left(4^{nN}-N-1\right)+N}{4^{n}-1}.
\]
\end{proof}
The asymptotic values of the cumulants can be explicitly computed
from expression (\ref{eq:cumulants}).
\begin{cor}
The cumulants of the limit distribution are
\[
\kappa_{2n}^{\left(\infty\right)}=\lim_{N\to+\infty}\kappa_{2n}^{\left(N\right)}=\frac{B_{2n}}{2n}\frac{6^{2n}}{2^{2n}-1}.
\]
As a consequence, the limit distribution is not Gaussian.
\end{cor}

\begin{rem}
For large $n,$ the asymptotic cumulants of the limit law behave like
\[
\kappa_{2n}^{\left(\infty\right)}=\frac{B_{2n}}{2n}3^{2n}
\]
which correspond to those of the uniform distribution on the interval
$\left[-3,3\right].$ This fact could help us approximate the limit distribution $\lambda\left(z\right).$
\end{rem}

\section{Acknowledgments}
The authors thank Eric Rowland for interesting comments on an earlier version of this manuscript, and Terri, Sen, Valerio, Natalie, and the Calvert Suite for their support.

\end{document}